\documentclass{article}
\usepackage[english]{babel}
\usepackage[utf8]{inputenc}
\usepackage{johd}
\usepackage{mathptmx}
\usepackage{amsmath}
\usepackage{mathtools}
\usepackage{verbatim}
\usepackage{appendix}

\numberwithin{equation}{section}

\usepackage{amsmath,amsthm}       
\usepackage{mathrsfs}      
\usepackage{amssymb}       
\usepackage{amsfonts}      

\usepackage{cancel}
\usepackage{indentfirst}

\usepackage{graphicx}
\usepackage{subfigure}
\usepackage{wrapfig}  
\usepackage{multicol}
\usepackage{tikz}     

\usepackage[framemethod=TikZ]{mdframed}
\usepackage{color}
\usepackage{authblk}
\usepackage{textcomp}

\newtheorem{thm}{Theorem}[section]
\newtheorem{lemma}[thm]{Lemma}

\newtheorem{remark}[thm]{Remark}
\newtheorem{prop}[thm]{Proposition}

\author{Yiqing Lin}

\author{Kun Xu\footnote{Corresponding author. Email address: \url{1949101x_k@sjtu.edu.cn} (K. Xu).}}
\affil[1]{\small{School of Mathematical Sciences, Shanghai Jiao Tong University, 200240 Shanghai, China.}}

\title{Mean-field reflected BSDEs driven by a marked point process}

\date{\today}

\begin{document}

\maketitle

\begin{abstract}
In this paper, we study a class of mean-field reflected backward stochastic differential equations (MFRBSDEs) driven by a marked point process. Based on a g-expectation representation lemma, we give the existence and uniqueness of MFRBSDEs driven by a marked point process under Lipschitz generator conditions. Besides, the well-posedness of this kind of BSDEs with exponential growth generator and unbounded terminal is also provided by $\theta$-method.
\end{abstract}

{\bf Keywords:} Mean-field, reflected BSDEs, marked point process.  

\section{Introduction}
Pardoux and Peng \cite{pardoux1990adapted} first introduced backward stochastic differential equations (BSDEs) and obtained the existence and uniqueness for the case of Lipschitz continuous coefficients. From then on, many efforts have been made to relax the conditions on the generators for the well-posedness of adapted solutions. For instance, Kobylanski \cite{kobylanski2000backward} studied quadratic BSDEs for a bounded terminal value $\xi$ via an approximation procedure of the driver. Thereafter, the result was generated by Briand and Hu \cite{briand2006bsde, briand2008quadratic} for unbounded terminal value $\xi$ of some suitable exponential moments. In contrast, Tevzadze \cite{tevzadze2008solvability} proposed a fundamentally different approach by means of a fixed point argument. At the same time, BSDEs have received numerous developments in various fields of PDEs \cite{delbaen2015uniqueness}, mathematical
finance \cite{el1997backward}, stochastic optimal control \cite{yong1999stochastic} and etc. 

Moreover, El Karoui et al. \cite{Karoui1997ReflectedSO} studied reflected backward stochastic differential equations (RBSDEs) in order to solve an obstacle problem for PDEs. The solution $Y$ of a RBSDE is required to be above a given continuous process $L$, i.e.,
$$
L_t \leq Y_t=\xi+\int_t^T f\left(s, Y_s, Z_s\right) d s+K_T-K_t-\int_t^T Z_s d B_s, \quad t \in[0, T],
$$
where the solution $(Y, Z, K)$ satisfies the so-called flat-off condition (or, Skorokhod condition):
$$
\int_0^T\left(Y_t-L_t\right) d K_t=0,
$$
where $K$ is an increasing process. El Karoui et al. \cite{Karoui1997ReflectedSO} proved the solvability of RBSDE with Lipschitz $f$ and square integrable terminal $\xi$.
After that, the result of \cite{Karoui1997ReflectedSO} is generated by  Kobylanski et al. \cite{kobylanski2002reflected} for quadratic RBSDEs with bounded terminal values and bounded obstacles. Lepeltier and Xu \cite{lepeltier2007reflected} constructed the existence of a solution with unbounded terminal values, but still with a bounded obstacle. Bayrakstar and Yao \cite{bayraktar2012quadratic} studied the well-posedness of quadratic RBSDE under unbounded terminal and unbounded obstacles.  The extension to the case of reflected BSDEs with jumps can be found in e.g. \cite{dumitrescu2016reflected,dumitrescu2015optimal,dumitrescu2016generalized,essaky2008reflected,hamadene2003reflected,hamadene2016reflected}.

The generalizations of BSDEs from a Brownian framework to a setting with jumps have aroused a lot of attention. Li and Tang \cite{Tang_1994} and Barles, Buckdahn and Pardoux \cite{barles1997backward} obtained the well-posedness for Lipschitz BSDEs with jumps (BSDEJ). Since then, different kind of BSDEJs have been investigated by many researchers, see \cite{Becherer_2006,morlais2010new,antonelli2016solutions,cohen2015stochastic,kazi2015quadratic,Barrieu_2013,ngoupeyou2010optimisation,jeanblanc2012robust,karoui2016quadratic,kaakai2022utility}.

In particular, a class of exponential growth BSDEs driven by a random measure associated with a marked point process as follows is investigated by many researchers. 
\begin{equation}
Y_t= \xi + \int_t^T f\left(t, Y_s, U_s\right)dA_s-\int_t^T \int_E U_s(e) q(d sd e).
\end{equation}
Here $q$ is a compensated integer random measure corresponding to some marked point process $(T_n,\zeta_n)_{n\ge 0}$, and $A$ is  the dual predictable projection of the event counting process related to the marked point process, which is a continuous and increasing process. The well-posedness of BSDEs driven by general marked point processes were investigated in Confortola \& Fuhrman \cite{Confortola2013} for the weighted-$L^2$ solution, Becherer \cite{Becherer_2006} and Confortola \& Fuhrman \cite{Confortola_2014} for the $L^2$ case,  Confortola,  Fuhrman \& Jacod \cite{Confortola2016} for the $L^1$ case and Confortola \cite{Confortola_2018} for the $L^p$ case. A more general BSDE with both Brownian motion diffusion term and a very general marked point process, which is non-explosive and has totally inaccessible jumps was studied in Foresta \cite{foresta2021optimal}. 

Motivated by the connection to control of McKean-Vlasov equation or mean-field games, see Carmona et al. \cite{carmona2013control} and Acciaio et al. \cite{acciaio2019extended}, mean-field BSDEs were introduced by Buckdahn, Djehiche, Li, Peng \cite{buckdahn2009mean} and Buckdahn, Li, Peng \cite{buckdahn2009mean1}. After that, reflected mean-field BSDEs has been considered by Li \cite{li2014reflected} and Djehiche et al. \cite{djehiche2019mean}. In this paper, we study a kind of mean-field type reflected BSDEs driven by a marked point process, where the mean-field interaction in terms of the distribution of the $Y$-component of
the solution enters both the driver and the obstacle (similar with \cite{djehiche2019mean}):
\begin{equation}
\label{reflected BSDE}
\left\{\begin{array}{l}
Y_t=\xi+\int_t^T f(s, Y_s, U_s, \mathbb{P}_{Y_s}) d A_s +\int_t^T d K_s-\int_t^T \int_E U_s(e) q(d s, d e), \quad 0 \leq t \leq T, \quad \mathbb{P} \text {-a.s. } \\
Y_t \geq h(t, Y_t, \mathbb{P}_{Y_t}), \quad 0 \leq t \leq T,\quad \mathbb{P} \text {-a.s. } \\
\int_0^T\left(Y_{t^{-}}-h\left(t^{-}, Y_{t^{-}}, \mathbb{P}_{Y_{t^{-}}}\right)\right) d K_t=0, \quad \mathbb{P} \text {-a.s. }
\end{array}\right.
\end{equation}

The rest of the paper is organized as follows. In section \ref{section pre}, we introduce some notations and basic results on BSDEs driven by marked point process. In section \ref{section Lip}, we investigate the solvability of (\ref{reflected BSDE}) under Lipschitz conditions on the generator. Section \ref{section exp} is devoted to provide the well-posedness of (\ref{reflected BSDE}) with exponential growth generator and unbounded terminal.

\section{Preliminaries}
\label{section pre}
In this section, we introduce some notions about marked point processes and some basic assumptions. More details about marked point processes can be found in \cite{foresta2021optimal, Bremaud1981, last1995marked, cohen2012existence}. 

In this paper we assume that $(\Omega, \mathscr{F}, \mathbb{P})$ is a complete probability space and $E$ is a Borel space. We call $E$ the mark space and  $\mathscr{E}$ is its Borel $\sigma$-algebra. Given a sequence of random variables $(T_n,\zeta_n)$ taking values in $[0,\infty]\times E$, set $T_0=0$ and $\mathbb P-a.s.$ 
\begin{itemize}
\item $T_n\le T_{n+1},\ \forall n\ge 0;$
\item $T_n<\infty$ implies $T_n<T_{n+1} \ \forall n\ge 0.$
\end{itemize}
The sequence $(T_n,\zeta_n)_{n\ge 0}$ is called a marked point process (MPP). Moreover, we assume the marked point process is non-explosive, i.e., $T_n\to\infty,\ \mathbb P-a.s.$

Define a random discrete measure $p$ on $((0,+\infty) \times E, \mathscr{B}((0,+\infty) \otimes \mathscr{E})$ associated with each MPP:
\begin{equation}
\label{eq p}
    p(\omega, D)=\sum_{n \geq 1} \mathbf{1}_{\left(T_n(\omega), \zeta_n(\omega)\right) \in D} .
\end{equation} 
For each $\tilde C \in \mathscr{E}$, define the counting process $N_t(\tilde C)=p((0, t] \times \tilde C)$ and denote $N_t=N_t(E)$. Obviously, both are right continuous increasing process starting from zero.  Define for $t \geq 0$
$$
\mathscr{G}_t^0=\sigma\left(N_s(\tilde C): s \in[0, t], \tilde C \in \mathscr{E}\right)
$$
and $\mathscr{G}_t=\sigma\left(\mathscr{G}_t^0, \mathscr{N}\right)$, where $\mathscr{N}$ is the family of $\mathbb{P}$-null sets of $\mathscr{F}$. 
Note by $\mathbb{G}=\left(\mathscr{G}_t\right)_{t \geq 0}$ the  completed filtration generated by the MPP, which is right continuous and satisfies the usual hypotheses.  Given a standard Brownian motion $W\in \mathbb R^d$, independent with the MPP, let $\mathbb F=(\mathcal F_t)$ be the completed filtration generated by the MPP and $W$, which satisfies the usual conditions as well.

Each marked point process has a unique compensator $v$, a predictable random measure such that
$$
\mathbb{E}\left[\int_0^{+\infty} \int_E C_t(e) p(d t d e)\right]=\mathbb{E}\left[\int_0^{+\infty} \int_E C_t(e) v(d t d e)\right]
$$
for all $C$ which is non-negative and $\mathscr{P}^{\mathscr{G}} \otimes \mathscr{E}$-measurable, where $\mathscr{P}^{\mathscr{G}}$ is the $\sigma$-algebra generated by $\mathscr{G}$-predictable processes. Moreover, in this paper we always assume that there exists a function $\phi$ on $\Omega \times[0,+\infty) \times \mathscr{E}$ such that $\nu(\omega, d t d e)=\phi_t(\omega, d e) d A_t(\omega)$, where $A$ is the dual predictable projection of $N$. In other words, $A$ is the unique right continuous increasing process with $A_0=0$ such that, for any non-negative predictable process $D$, it holds that,
$$
\mathbb E\left[\int_{0}^\infty D_t dN_t\right]=E\left[\int_{0}^\infty D_t dA_t\right].
$$

Fix a terminal time $T>0$, we can define the integral
$$
\int_0^T \int_E C_t(e) q(d t d e)=\int_0^T \int_E C_t(e) p(d t d e)-\int_0^T \int_E C_t(e) \phi_t(d e) d A_t,
$$
under the condition
$$
\mathbb{E}\left[\int_0^T \int_E\left|C_t(e)\right| \phi_t(d e) d A_t\right]<\infty .
$$
Indeed, the process $\int_0^{\cdot} \int_E C_t(e) q(d t d e)$ is a martingale. Note that  $\int_a^b$ denotes an integral on $(a, b]$ if $b<\infty$, or on $(a, b)$ if $b=\infty$.

The following spaces are frequently called in the sequel.
\begin{itemize}
\item $\mathbb{L}^0$ denotes the space of all real-valued, $\mathcal{G}_T$-measurable random variables.
\item $\mathbb{L}^p :=\left\{\xi \in \mathbb{L}^0 :\|\xi\|_p :=\left\{E\left[|\xi|^p\right]\right\}^{\frac{1}{p}}<\infty\right\}$, for all $p \in[1, \infty)$.
\item $\mathbb{L}^{\infty}:=\left\{\xi \in \mathbb{L}^0:\|\xi\|_{\infty} := {esssup}_{\omega \in \Omega}|\xi(\omega)|<\infty\right\}$.
\item  $\mathbb{C}^0$ denotes the set of all real-valued continuous  process adapted to $\mathbb{G}$ on $[0, T]$. 
\item Let $\mathbb{K}$ be the subset of $\mathbb{C}^0$ that consists of all real-valued increasing and continuous adapted process starting from 0, and $\mathbb K^p$ is a subset of $\mathbb{K}$ such that for each $X\in \mathbb K^p$, $X_T\in \mathbb{L}^p$.
\item  ${S}^0$ denotes the set of real-valued, adapted and c\`adl\`ag processes $\left\{Y_t\right\}_{t \in[0, T]}$. 
\item For any $\left\{\ell_t\right\}_{t \in[0, T]} \in  S^0$, define $\ell_*^{ \pm} := \sup _{t \in[0, T]}\left(\ell_t\right)^{ \pm}$. Then
$$
\ell_* := \sup _{t \in[0, T]}\left|\ell_t\right|=\sup _{t \in[0, T]}\left(\left(\ell_t\right)^{-} \vee\left(\ell_t\right)^{+}\right)=\sup _{t \in[0, T]}\left(\ell_t\right)^{-} \vee \sup _{t \in[0, T]}\left(\ell_t\right)^{+}=\ell_*^{-} \vee \ell_*^{+} .
$$
\item For any real $p \geq 1,\ {S}^p$ denotes the set of real-valued, adapted and c\`adl\`ag processes $\left\{Y_t\right\}_{t \in[0, T]}$ such that,
$$
\|Y\|_{{S}^p}:=\mathbb{E}\left[\sup _{0 \leq t \leq T}\left|Y_t\right|^p\right]^{1 / p}<+\infty.
$$
Then $\left({S}^p,\|\cdot\|_{ {S}^p}\right)$ is a Banach space.
\item ${S}^{\infty}$ is the space of $\mathbb{R}$-valued càdlàg and $\mathcal{G}_t$-progressively measurable processes $Y$ such that
$$
\|Y\|_{{S}^\infty}:=\sup _{0 \leq t \leq T}\left\|Y_t\right\|_{\infty}<+\infty.
$$
\item For any $p \geq 1$, we denote by $\mathcal{E}^p$ the collection of all stochastic processes $Y$ such that $e^{|Y|} \in$ $S^p$.   We write $Y \in \mathcal{E}$ if $Y \in \mathcal{E}^p$ for any $p \geq 1$.
\item $\mathbb{H}^p$ is the space of real-valued and $\mathcal{G}_t$-progressively measurable processes $Z$ such that
$$
\|Z\|_{\mathbb{H}^p}^p:=\mathbb{E}\left[\left(\int_0^T\left|Z_t\right|^2 d t\right)^{\frac{p}{2}}\right]<+\infty .
$$
\item $L^0\left(\mathscr{B}(E)\right)$ denotes the space of $\mathcal B(E)$-measurable functions.  For $u \in L^0\left(\mathscr{B}(E)\right)$, define
$$
L^2(E,\mathcal{B}(E),\phi_t(\omega,dy)):=\left\{\left\|u\right\|_t:=\left(\int_E\left|u(e)\right|^2  \phi_t(d e)\right)^{1 / 2}<\infty\right\}.
$$
\item $H^{2,p}_{\nu}$ is the space of predictable processes $U$ such that 
$$
\|U\|_{H_\nu^{2,p}}:=\left(\mathbb{E}\left[\int_{[0, T]} \int_E\left|U_s(e)\right|^2 \phi_s(de) d A_s\right ]^{\frac{p}{2}}\right)^{\frac{1}{p}}<\infty.
$$
\item $H^{2,loc}_{\nu}$ is the space of predictable processes $U$ such that 
$$
\|U\|_{H_\nu^{2,loc}}:=\int_{[0, T]} \int_E\left|U_s(e)\right|^2 \phi_s(de) d A_s<\infty,\ \mathbb P-a.s.
$$
As in \cite{Confortola2013}, we say that $U, U^{\prime} \in H^{2,p}_{\nu}$ (respectively, $U, U^{\prime} \in H^{2,loc}_{\nu}$ ) are equivalent if they coincide almost everywhere with respect to the measure $\phi_t(\omega, d y) d A_t(\omega) \mathbb{P}(d \omega)$  and this happens if and only if $\left\|U-U^{\prime}\right\|_{H_{\nu}^{2,p}}=0$ (equivalently, $\left\|U-U^{\prime}\right\|_{H_{\nu}^{2,loc}}=0,\ \mathbb P-a.s.$). With a little abuse of notation, we  still denote $H^{2,p}_{\nu}$ (respectively, $H^{2,loc}_{\nu}$) the corresponding set of equivalence classes, endowed with the norm $\|\cdot\|_{H_\nu^{2,p}}$ (respectively, $\|\cdot\|_{H_\nu^{2,loc}} $). In addition, $H_{\nu}^{2,p}$ is a Banach space. 
\item $\mathbb{J}^{\infty}$ is the space of functions such that
$$
\|\psi\|_{\mathbb{J}^\infty}:=\left\| \left\|\psi\right\|_{{L}^{\infty}(\nu(\omega))}\right\|_{\infty}<\infty.
$$
\item $\mathcal S_{0,T}$ denotes the collection of $\mathbb{G}$-stopping times $\tau$ such that $0\le\tau\le T,\ \mathbb P-a.s.$. For any $\tau\in\mathcal S_{0,T}$, $\mathcal S_{\tau,T}$ denotes the collection of $\mathbb{G}$-stopping times $\tilde\tau$ such that $\tau\le\tilde\tau\le T,\ \mathbb P-a.s.$
\end{itemize}

For $u \in L^0\left(\mathscr{B}(E)\right)$ and $\lambda >0$, we introduce the positive predictable process denoted by
$$
j_\lambda (\cdot,u)=\int_E (\mathrm{e}^{\lambda  u(e)}-1-\lambda  u(e))\phi_{\cdot}(de).
$$

For $\beta>0$, we introduce the following spaces.
\begin{itemize}
    \item $\mathbb{L}_\beta^p$ is the set of real-valued càdlàg adapted process $Y$ such that $\|Y\|_{\mathbb{L}_\beta^p}^p = \sup_{0 \leq t \leq T} \mathbb{E}\left[e^{\beta A_s}Y_s^{p}\right]$. $\mathbb{L}_\beta^p$ is a Banach space. We set $\mathbb{L}^p= \mathbb{L}_0^p$.

    \item $L^{r, \beta}(A)$ is the space of all $\mathbb{F}$-progressive processes $X$ such that
$$
\|X\|_{L^{r, \beta}(A)}^r=\mathbb{E}\left[\int_0^T e^{\beta A_s}\left|X_s\right|^r d A_s\right]<\infty.
$$
    \item $L^{r, \beta}(p)$ is the space of all $\mathbb{F}$-predictable processes $U$ such that
$$
\|U\|_{L^{r, \beta}(p)}^r=\mathbb{E}\left[\int_0^T \int_E e^{\beta A_s}\left|U_s(e)\right|^r \phi_s(d e) d A_s\right]<\infty.
$$
    \item $L^{r, \beta}\left(W, \mathbb{R}^d\right)$ is the space of $\mathbb{F}$-progressive processes $Z$ in $\mathbb{R}^d$ such that
$$
\|Z\|_{L^{r, \beta}(W)}^r=\mathbb{E}\left[\int_0^T e^{\beta A_s}\left|Z_s\right|^r d s\right]<\infty.
$$
    \item $S^2_*$ is the space of all $\mathbb{F}$-progressive processes $Y$ such that
$$
\|Y\|_{S^2_*}^2 = \sup_{0 \leq t \leq T} \mathbb{E} \left[  Y_{t}^2 \right] < \infty.
$$
    \item $\mathscr{A}_D$ is the space of all càdlàg non-decreasing deterministic processes $K$ starting from the origin, i.e. $K_0 = 0$.
\end{itemize}

Next, we introduce some basic assumptions that run through the whole paper.

\hspace*{\fill}\\
\hspace*{\fill}\\
\noindent(\textbf{H1}) The process $A$ is continuous with $\|A_T\|_\infty<\infty$.
\hspace*{\fill}\\
\hspace*{\fill}\\

The first assumption is on the dual predictable projection $A$ of the counting process $N$ relative to $p$. We would like to emphasize that for $A_t$, we do not require absolute continuity with respect to the Lebesgue measure.

\hspace*{\fill}\\
\hspace*{\fill}\\
\hspace*{\fill}\\
\noindent(\textbf{H2}) For every $\omega \in \Omega,\ t \in[0, T],\ r \in \mathbb{R}$, $\mu \in \mathcal{P}_p(\mathbb{R})$ the mapping
$
f(\omega, t, r, \cdot, \mu):L^2(\mathcal{B}(E),\nu_t;\mathbb R) \rightarrow \mathbb{R}
$
satisfies:
for every $U \in {H_\nu^{2,2}}$,
$$
(\omega, t, r, \mu) \mapsto f\left(\omega, t, r, U_t(\omega, \cdot), \mu \right)
$$
is Prog $\otimes \mathscr{B}(\mathbb{R})$-measurable.

\hspace*{\fill}\\
\hspace*{\fill}\\
\noindent(\textbf{H3})

\textbf{(a) (Continuity condition)} 
For every $\omega \in \Omega, t \in[0, T], y \in \mathbb{R}$, $u\in L^2(\mathcal{B}(E),\nu_t; \mathbb R)$, $\mu \in \mathcal{P}_p(\mathbb{R})$,  $(y, u, \mu) \longrightarrow f(t, y, u, \mu)$ is continuous, the process $\left(f\left(t, 0,0,0, \delta_0\right)\right)_{t \leq T}$ is $\mathcal{P}$-measurable and  
$$
\mathbb{E}\left[\int_0^T e^{\beta A_s}|f(s, 0, 0, \delta_0)|^2 d A_s\right]<\infty.
$$
\hspace*{\fill}\\
\hspace*{\fill}\\

\textbf{(b) (Lipschitz condition)}  
There exists $\beta\geq 0$, such that for every $\omega \in \Omega,\ t \in[0, T],\ y_1, y_2 \in \mathbb{R}$,\ $u\in L^2(E,\mathcal{B}(E),\phi_t(\omega,dy))$, $\mu_1, \mu_2 \in \mathcal{P}_p(\mathbb{R})$, we have
$$
\begin{aligned}
& \left|f(\omega, t, y_1, u(\cdot), \mu_1)-f\left(\omega, t, y_2, u(\cdot), \mu_2\right)\right| \leq \beta\left(\left|y_1-y_2\right|+ \mathcal{W}_1\left(\mu_1, \mu_2\right)\right).
\end{aligned}
$$
\hspace*{\fill}\\

\textbf{(c) (Growth condition)}
For all $t \in[0, T], \ (y,u) \in \mathbb{R} \times  L^2(\mathcal{B}(E),\nu_t; \mathbb R):\ \mathbb{P}$-a.s, there exists $\lambda>0$ such that,
$$
\underline{q}(t, y, u)=-\frac{1}{\lambda} j_{\lambda}(t,- u)-\alpha_t-\beta\left(|y|+ \mathcal{W}_1\left(\mu, \delta_0\right)\right) \leq f(t, y, u, \mu) \leq \frac{1}{\lambda} j_{\lambda}(t, u)+\alpha_t+\beta\left(|y|+\mathcal{W}_1\left(\mu, \delta_0\right)\right)=\bar{q}(t, y, u),
$$
where $\{\alpha_t\}_{0 \leq t \leq T}$ is  a progressively measurable nonnegative stochastic process.
\hspace*{\fill}\\
\hspace*{\fill}\\

\textbf{(d) (Integrability condition)}
We assume necessarily,
$$
\forall p>0, \quad \mathbb{E}\left[\exp \left\{p \lambda e^{\beta A_T}(|\xi|\vee h_*)+p\lambda  \int_0^T e^{\beta A_s}\alpha_s d A_s\right\}+\int_0^T\alpha_s^2dA_s\right]<\infty.
$$
\hspace*{\fill}\\

\textbf{(e) (Convexity/Concavity condition)}  
For each $(t, y) \in[0, T] \times \mathbb{R}$, $u\in L^2(\mathcal{B}(E),\nu_t; \mathbb R), \ \mu \in \mathcal{P}_p(\mathbb{R}), \ u \rightarrow f(t, y, u, \mu)$ is convex or concave.

\hspace*{\fill}\\
\hspace*{\fill}\\
\noindent(\textbf{H4})
$h$ is a mapping from $[0, T] \times \Omega \times \mathbb{R} \times \mathcal{P}_p(\mathbb{R})$ into $\mathbb{R}$ such that
\hspace*{\fill}\\
\hspace*{\fill}\\

\textbf{(a) (Continuity condition)} 
For all $(y, \mu) \in \mathbb{R} \times \mathcal{P}_p(\mathbb{R}),(h(t, y, \mu))_t$ is a right-continuous left-limited process the process.
\hspace*{\fill}\\
\hspace*{\fill}\\

\textbf{(b) (Lipschitz condition in $y$ and $\mu$)}
$h$ is Lipschitz w.r.t. $(y, \mu)$ uniformly in $(t, \omega)$, i.e. there exists two positive constants $\gamma_1$ and $\gamma_2$ such that $\mathbb{P}$-a.s. for all $t \in[0, T]$,
$$
\left|h\left(t, y_1, \mu_1\right)-h\left(t, y_2, \mu_2\right)\right| \leq \gamma_1\left|y_1-y_2\right|+\gamma_2 \mathcal{W}_1\left(\mu_1, \mu_2\right)
$$
for any $y_1, y_2 \in \mathbb{R}$ and $\mu_1, \mu_2 \in \mathcal{P}_p(\mathbb{R})$.
\hspace*{\fill}\\
\hspace*{\fill}\\

\textbf{(c)}
The final condition $\xi: \Omega \rightarrow \mathbb{R}$ is $\mathscr{F}_T$-measurable and satisfies $\xi \geq h\left(T, \xi, \mathbb{P}_{\xi}\right)$;

\section{Lipschitz case}
\label{section Lip}
\noindent\textbf{(H3')}
same as (H3) but (b) is replaced by:
$f$ is Lipschitz w.r.t. $(y, u, \mu)$ uniformly in $(t, \omega)$, i.e. there exists a positive constant $C_f$ such that $\mathbb{P}$-a.s. for all $t \in[0, T]$,
$$
\left|f\left(t, y_1, u_1, \mu_1\right)-f\left(t, y_2, u_2, \mu_2\right)\right| \leq C_f\left(\left|y_1-y_2\right|+\left(\int_E\left|u_1(e)-u_2(e)\right|^2 \phi_t(\omega, d e)\right)^{1 / 2}+\mathcal{W}_p\left(\mu_1, \mu_2\right)\right)
$$
for any $y_1, y_2 \in \mathbb{R}, u_1, u_2 \in L^{r, \beta}(p)$ and $\mu_1, \mu_2 \in \mathcal{P}_p(\mathbb{R})$.\\
\hspace*{\fill}\\

\noindent\textbf{(H4')}
same as (H4) but (b) is replaced by:
$h$ is Lipschitz w.r.t. $(y, \mu)$ uniformly in $(t, \omega)$, i.e. there exists two positive constants $\gamma_1$ and $\gamma_2$ such that $\mathbb{P}$-a.s. for all $t \in[0, T]$,
$$
\left|h\left(t, y_1, \mu_1\right)-h\left(t, y_2, \mu_2\right)\right| \leq \gamma_1\left|y_1-y_2\right|+\gamma_2 \mathcal{W}_p\left(\mu_1, \mu_2\right)
$$
for any $y_1, y_2 \in \mathbb{R}$ and $\mu_1, \mu_2 \in \mathcal{P}_p(\mathbb{R})$.
\hspace*{\fill}\\

We first give a prior estimate about BSDE driven by a marked point process.

\begin{prop}
\label{estimate BSDE}
Let $T>0$ and let $f^1$ be a Lipschitz driver with Lipschitz constant $C$ and let $f^2$ be a driver. For $i=1,2$, let $\left(Y^i, U^i\right)$ be a solution of the BSDE associated to:
$$
Y_t^i=\xi^i+\int_t^T f(s, Y_s^i, U_s^i) d A_s -\int_t^T \int_E U_s^i(e) q(d s, d e).
$$
For $s$ in $[0, T]$, denote $\bar{Y}_s:=Y_s^1-Y_s^2, \bar{U}_s:=U_s^1-U_s^2, \bar{\xi} := \xi^1-\xi^2$, and $\bar{f}(s):=f^1\left(s, Y_s^2, U_s^2\right)-f^2\left(s, Y_s^2, U_s^2\right)$. Let $\eta, \beta>0$ be such that $2\beta \geq \frac{2}{\eta}+2 C$. If $\eta \leq \frac{1}{C^2}$, then, for each $t \in[0, T]$, we have
$$
|e^{\beta A_t} \bar{Y}_t|^2 \leq \mathbb{E}_t\left[|e^{\beta A_T} \bar{\xi}|^2 \right]+ \eta \mathbb{E}_t\left[\int_t^T |e^{\beta A_s} \bar{f}(s)|^2 d A_s\right] \quad a.s.
$$
and
$$
|e^{\beta A_t} \bar{Y}_t|^p \leq 2^{p/2-1}\mathbb{E}_t\left[|e^{\beta A_T} \bar{\xi}|^p \right]+ \eta^{p/2} \mathbb{E}_t\left[\left(\int_t^T |e^{\beta A_s} \bar{f}(s)|^2 d A_s\right)^{p/2} \right] \quad a.s.
$$
\end{prop}

\begin{proof}
From Itô's formula applied to the semimartingale $e^{\beta A_s} \bar{Y}_s$ between $t$ and $T$ and taking the conditional expectation given $\mathcal{F}_t$, it follows
$$
\begin{aligned}
&(e^{\beta A_t} \bar{Y}_t)^2  + \mathbb{E}_t \left[2\beta \int_t^T (e^{\beta A_s} \bar{Y}_s)^2 d A_s+\int_t^T\int_E (e^{\beta A_s} \bar{U}_s(e))^2 \phi_s(de)d A_s\right] \\
&=  2 \mathbb{E}_t \left[\int_t^T e^{2\beta A_s} \bar{Y}_s\left(f^1\left(s, Y_s^1, U_s^1\right)-f^2\left(s, Y_s^2, U_s^2\right)\right) d A_s\right] + \mathbb{E}_t\left[(e^{\beta A_t} \bar{\xi})^2 \right].
\end{aligned}
$$
Moreover,
$$
\begin{aligned}
\left|f^1\left(s, Y_s^1, U_s^1\right)-f^2\left(s, Y_s^2, U_s^2\right)\right| & \leq\left|f^1\left(s, Y_s^1, U_s^1\right)-f^1\left(s, Y_s^2, U_s^2\right)\right|+\left|\bar{f}_s\right| \\
& \leq C\left|\bar{Y}_s\right|+\left(C\left\|\bar{U}_s\right\|_\nu+\left|\bar{f}_s\right|\right) .
\end{aligned}
$$
Now, for all real numbers $y, z, k, f$ and $\varepsilon>0$
$$
2 y(C u+f) \leq \frac{y^2}{\varepsilon^2}+\varepsilon^2(C u+f)^2 \leq \frac{y^2}{\varepsilon^2}+2 \varepsilon^2\left(C^2 u^2+f^2\right) \text {. Hence, we }
$$
get
$$
\begin{aligned}
(e^{\beta A_t} & \bar{Y}_t)^2+\mathbb{E}_t\left[2\beta \int_t^T e^{\beta A_s} \bar{Y}_s^2 d A_s+\int_t^T\int_E (e^{\beta A_s} \bar{U}_s(e))^2 \phi_s(de)d A_s \right] \\
\leq & \mathbb{E}_t\left[\left(2 C+\frac{1}{\varepsilon^2}\right) \int_t^T (e^{\beta A_s} \bar{Y}_s)^2 d A_s+2 C^2 \varepsilon^2 \int_t^T (e^{\beta A_s}\left\|\bar{U}_s\right\|_\nu)^2d A_s \right] \\
& +2 \varepsilon^2 \mathbb{E}_t\left[\int_t^T (e^{\beta s} \bar{f}_s)^2 d A_s \right] + \mathbb{E}_t\left[(e^{\beta A_t} \bar{\xi})^2 \right] .
\end{aligned}
$$
Let us make the change of variable $\eta=2 \epsilon^2$. Then, for each $\beta, \eta>0$ chosen as in the proposition, we have
$$
\left|e^{\beta t} \bar Y_t\right|^2 \leq \mathbb{E}_t\left[\left|e^{\beta \tau} \bar \xi\right|^2 \right]+\eta \mathbb{E}_t\left[\left(\int_t^\tau\left|e^{\beta s} \bar f_s\right|^2 d s\right) \right] \text { a.s. }
$$
from which it follows that
$$
\left|e^{\beta t} \bar Y_t\right|^p \leq\left(\mathbb{E}_t\left[\left|e^{\beta \tau} \bar \xi\right|^2 \right]+\eta \mathbb{E}_t\left[\left(\int_t^\tau\left|e^{\beta s} \bar f_s\right|^2 d s\right) \right]\right)^{p / 2} \text { a.s. },
$$
which leads to, by convexity relation and Hölder inequality,
$$
\left|e^{\beta t} \bar Y_t\right|^p \leq 2^{p / 2-1}\left(\mathbb{E}_t\left[\left|e^{\beta \tau} \bar \xi\right|^p \right]+\eta^{\frac{p}{2}} \mathbb{E}_t\left[\left(\int_t^\tau\left|e^{\beta s} \bar f_s\right|^2 d s\right)^{p / 2} \right]\right) \text { a.s.. }
$$
\end{proof}

Let $\widehat{\Phi}: \mathbb{L}_\beta^p \longrightarrow \mathbb{L}_\beta^p$ be the mapping that associates to a process $Y$ another process $\widehat{\Phi}(Y)$ defined by:
$$
\widehat{\Phi}(Y)_t=\underset{\tau \text { stopping time } \geq t}{\operatorname{ess} \sup } \mathcal{E}_{t, \tau}^{f\circ Y}\left[\xi \mathbf{1}_{\{\tau=T\}}+h(\tau, Y_\tau,\mathbb{P}_{Y_s\mid s = \tau}) \mathbf{1}_{\{\tau<T\}}\right], \quad \forall t \leq T,
$$
where where $(f \circ Y)(t, y, u)=f\left(t, y, u, \mathbb{P}_{Y_t}\right)$ and $\mathcal{E}_{t, \tau}^{f\circ Y}\left[\xi \mathbf{1}_{\{\tau=T\}}+h(\tau, Y_\tau,\mathbb{P}_{Y_s\mid s = \tau}) \mathbf{1}_{\{\tau<T\}}\right]:=y_t^\tau$ is the solution to the following standard BSDE:
\begin{equation}
\label{standard BSDE}
y_t^\tau=\xi \mathbf{1}_{\{\tau=T\}}+h_\tau(Y_\tau,\mathbb{P}_{Y_s\mid s = \tau}) \mathbf{1}_{\{\tau<T\}}+\int_t^\tau f\left(s, y_s^\tau, u_s^\tau, \mathbb{P}_{Y_s\mid s = \tau})\right) d A_s-\int_t^\tau \int_E u_s^\tau(e) q(ds, de). 
\end{equation}

Next we introduce a $g$-expectation representation lemma based on comparison theorem which plays an important role in our proof.
\begin{lemma}[\cite{2023arXiv231020361L} lemma 4.1]
\label{representation}
Under assumptions (H1)-(H4), suppose for each $p\ge 1$, $(Y, U, K) $ is a solution to the reflected BSDE:
\begin{equation}
\label{BSDE with B}
\left\{\begin{array}{l}
Y_t=\xi+\int_t^T f\left(s, Y_s, U_s\right) d A_s +\int_t^T g\left(s, Y_s, Z_s\right) d s+\int_t^T d K_s-\int_t^T Z_s d B_s-\int_t^T \int_E U_s(e) q(d s, d e), \quad 0 \leq t \leq T, \quad \mathbb{P} \text {-a.s. } \\
Y_t \geq h(t), \quad 0 \leq t \leq T,\quad \mathbb{P} \text {-a.s. } \\
\int_0^T\left(Y_{s^{-}}-h(s^{-})\right) d K_s=0, \quad \mathbb{P} \text {-a.s. }
\end{array}\right.
\end{equation}
Then,
$$
Y_t=\underset{\tau \in \mathcal{S}_{t,T}}{\operatorname{ess} \sup } \ \mathcal{E}_{t, \tau}^f\left[\xi \mathbf{1}_{\{\tau=T\}}+h_\tau \mathbf{1}_{\{\tau<T\}}\right] .
$$
\end{lemma}

\begin{remark}
\label{space of y}
     \textcolor{red}{From \cite{foresta2021optimal} we can build a contraction mapping on $L^{2, \beta}(A) \cap$ $L^{2, \beta}(W) \times L^{2, \beta}(p) \times L^{2, \beta}(W) \times \mathcal{A}_D$, but acually, by proposition 4.2 in \cite{2023arXiv231020361L}, we can deduce that \textcolor{red}{$Y \in \mathcal{E}$}.}
\end{remark}

\begin{thm}
\label{thm lip}
Suppose that Assumption (H1),(H2),(H3'),(H4') and (H5) are satisfied for some $p\geq 2$. Assume that $\gamma_1$ and $\gamma_2$ satisfy
$$
\gamma_1^p+\gamma_2^p<2^{1-p}.
$$
Then the system (\ref{reflected BSDE}) has a unique solution in $\mathcal{E} \times H_\nu^{2,p} \times \mathbb{K}^p$.
\end{thm}

\begin{proof}
The proof is organized in three steps. 

\textbf{Step 1.} We first show that $\widehat{\Phi}$ is a well-defined mapping. Indeed, let $\bar{Y} \in \mathbb{L}_\beta^p$. Since $h$ satisfies Assumption H(4), it follows that $h\left(t, \bar{Y}_t, \mathbb{P}_{\bar{Y}_t}\right)$ satisfies Assumption H(3) (d). Therefore, from \cite{foresta2021optimal} and Remark \ref{space of y} we know that there exists an unique solution $(\hat{Y}, \hat{U}, \hat{K}) \in \mathcal{E} \times \mathcal{H}_\nu^p \times \mathbb{K}^p$ of the reflected BSDE associated with obstacle $h\left(t, \bar{Y}_t, \mathbb{P}_{\bar{Y}_t}\right)$, terminal condition $\xi$ and driver $(f \circ \bar{Y})(t, \omega, y, u)$. By the comparison results in \cite{2023arXiv231020361L} we can build connection between the $Y$-component of the solution of a reflected BSDE and optimal stopping with nonlinear expectations (see Lemma \ref{representation}), we get that $\widehat{\Phi}(\bar{Y})=\hat{Y} \in \mathcal{E} \subset \mathbb{L}_\beta^p$.

\textbf{Step2.} We show that $\widehat{\Phi} : \mathbb{L}_\beta^p \longrightarrow \mathbb{L}_\beta^p$ is a contraction mapping in $[T-h, T]$. 
First, note that by the Lipschitz continuity of $f$ and $h$, for $Y, \bar{Y} \in \mathbb{L}_\beta^p $ and $U, \bar{U} \in \mathcal{H}_\nu^p$,
\begin{equation}
\begin{aligned}
\label{lip condition}
& \left|f\left(s, Y_s, U_s, \mathbb{P}_{Y_s}\right)-f\left(s, \bar{Y}_s, \bar{U}_s, \mathbb{P}_{\bar{Y}_s}\right)\right| \leq C_f\left(\left|Y_s-\bar{Y}_s\right|+\left|U_s-\bar{U}_s\right|_\nu+\mathcal{W}_p\left(\mathbb{P}_{Y_s}, \mathbb{P}_{\bar{Y}_s}\right)\right), \\
& \left|h\left(s, Y_s, \mathbb{P}_{Y_s}\right)-h\left(s, \bar{Y}_s, \mathbb{P}_{\bar{Y}_s}\right)\right| \leq \gamma_1\left|Y_s-\bar{Y}_s\right|+\gamma_2 \mathcal{W}_p\left(\mathbb{P}_{Y_s}, \mathbb{P}_{\bar{Y}_s}\right).
\end{aligned}
\end{equation}
For the $p$-Wasserstein distance, we have the following inequality: for $0 \leq s \leq u \leq t \leq T$,
$$
\sup _{u \in[s, t]} \mathcal{W}_p\left(\mathbb{P}_{Y_u}, \mathbb{P}_{\bar{Y}_u}\right) \leq \sup _{u \in[s, t]}\left(\mathbb{E}\left[\left|Y_u-\bar{Y}_u\right|^p\right]\right)^{\frac{1}{p}}
$$
from which we derive the following useful inequality
$$
\sup _{u \in[s, t]} \mathcal{W}_p\left(\mathbb{P}_{Y_u}, \delta_0\right) \leq \sup _{u \in[s, t]}\left(\mathbb{E}\left[\left|Y_u\right|^p\right]\right)^{\frac{1}{p}}.
$$
Fix $Y, \bar{Y} \in \mathbb{L}_\beta^p$. For any $t \leq T$, by Proposition \ref{estimate BSDE}, we have
$$
\begin{aligned}
& \left| \widehat{\Phi}(Y)_t -\widehat{\Phi}(\bar{Y})_t \right|^p\\
&= \underset{\tau \in \mathcal{T}_t}{\operatorname{ess} \sup } \ \left| \mathcal{E}_{t, \tau}^{f\circ Y}\left[\xi \mathbf{1}_{\{\tau=T\}}+h(\tau, Y_\tau,\mathbb{P}_{Y_s\mid s = \tau}) \mathbf{1}_{\{\tau<T\}}\right] -  \mathcal{E}_{t, \tau}^{f\circ \bar{Y}}\left[\xi \mathbf{1}_{\{\tau=T\}}+h(\tau, \bar{Y}_\tau,\mathbb{P}_{\bar{Y}_s\mid s = \tau}) \mathbf{1}_{\{\tau<T\}}\right]\right|^p \\
& \leq \underset{\tau \in \mathcal{T}_t}{\operatorname{ess} \sup } \ 2^{\frac{p}{2}-1}\mathbb{E}_t\left[\eta^{\frac{p}{2}} \left(\int_t^\tau e^{2 \beta(A_s-A_t)}\left|(f \circ Y) \left(s, \widehat{Y}_s^\tau, \widehat{U}_s^\tau \right)-(f \circ \bar{Y})\left(s, \widehat{Y}_s^\tau, \widehat{U}_s^\tau \right)\right|^2 d A_s\right)^{\frac{p}{2}}\right. \\
& \quad \left.+e^{p \beta(A_\tau-A_t)}\left|h\left(\tau, Y_\tau, \mathbb{P}_{Y_s \mid s=\tau}\right)-h\left(\tau, \bar{Y}_\tau, \mathbb{P}_{\bar{Y}_s \mid s=\tau}\right)\right|^p \right], \\
&= \underset{\tau \in \mathcal{T}_t}{\operatorname{ess} \sup } \ 2^{\frac{p}{2}-1}\mathbb{E}_t\left[\eta^{\frac{p}{2}} \left(\int_t^\tau e^{2 \beta(A_s-A_t)}\left|f\left(s, \widehat{Y}_s^\tau, \widehat{U}_s^\tau, \mathbb{P}_{Y_s} \right)-f\left(s, \widehat{Y}_s^\tau, \widehat{U}_s^\tau, \mathbb{P}_{\bar{Y}_s} \right)\right|^2 d A_s\right)^{\frac{p}{2}}\right. \\
& \quad \left.+e^{p\beta(A_\tau-A_t)}\left|h\left(\tau, Y_\tau, \mathbb{P}_{Y_s \mid s=\tau}\right)-h\left(\tau, \bar{Y}_\tau, \mathbb{P}_{\bar{Y}_s \mid s=\tau}\right)\right|^p \right], \\
\end{aligned}
$$
with $\eta, \beta>0$ such that $\eta \leq \frac{1}{C_f^2}$ and $2\beta \geq 2 C_f+\frac{2}{\eta}$, where $\left(\widehat{Y}^\tau, \widehat{Z}^\tau, \widehat{U}^\tau\right)$ is the solution of the BSDE associated with driver $f \circ \bar{Y}$ and terminal condition $h\left(\tau, \bar{Y}_\tau, \mathbb{P}_{\bar{Y}_s \mid s=\tau}\right) \mathbf{1}_{\{\tau<T\}}+\xi \mathbf{1}_{\{\tau=T\}}$. Therefore, using (\ref{lip condition}) and the fact that
$$
\mathcal{W}_p\left(\mathbb{P}_{Y_s \mid s=\tau}, \mathbb{P}_{\bar{Y}_s \mid s=\tau}\right) \leq \mathbb{E}\left[\left|Y_s-\bar{Y}_s\right|^p \right]_{\mid s=\tau} \leq \sup _{\tau \in \mathcal{T}_t} \mathbb{E}\left[\left|Y_\tau-\bar{Y}_\tau\right|^p\right]^{\frac{1}{p}},
$$
we have, for any $t \in [T-h, T]$,
$$
\begin{aligned}
&e^{p \beta A_t}\left|\widehat{\Phi}(Y)_t-\widehat{\Phi}(\bar{Y})_t\right|^p\\ 
&\leq \underset{\tau \in \mathcal{T}_t}{\operatorname{ess} \sup }\ 2^{\frac{p}{2}-1}\mathbb{E}_t\left[\int_t^\tau e^{p\beta A_s} \eta^{\frac{p}{2}} C_f^p \left(A_T-A_{T-h}\right)^{\frac{p-2}{p}}\mathbb{E}\left[\left|Y_s-\bar{Y}_s\right|^p\right] d A_s+e^{p \beta A_\tau}\left(\gamma_1\left|Y_\tau-\bar{Y}_\tau\right|+\gamma_2\left\{\mathbb{E}\left[\left|Y_s-\bar{Y}_s\right|^p\right]_{\mid s=\tau}\right\}^{1 / p}\right)^p \right]\\
& \leq \underset{\tau \in \mathcal{T}_t}{\operatorname{ess} \sup }  \ 2^{\frac{p}{2}-1}\mathbb{E}_t\left[\int_t^\tau e^{p \beta A_s} \eta^{\frac{p}{2}} C_f^p \left(A_T-A_{T-h}\right)^{\frac{p-2}{p}}\mathbb{E}\left[\left|Y_s-\bar{Y}_s\right|^p\right] d A_s +e^{p \beta A_\tau}\left\{2^{p-1} \gamma_1^p\left|Y_\tau-\bar{Y}_\tau\right|^p+2^{p-1} \gamma_2^p \mathbb{E}\left[\left|Y_s-\bar{Y}_s\right|^p\right]_{\mid s=\tau}\right\} \right]  \\
&\leq  \underset{\tau \in \mathcal{T}_t}{\operatorname{ess} \sup }  \ 2^{\frac{p}{2}-1} \\
&\quad \times \mathbb{E}_t \left[\eta^{\frac{p}{2}} C_f^p \left(A_T-A_{T-h}\right)^{\frac{p-2}{p}} \sup_{s \in [T-h ,T]}\mathbb{E}\left[\left|Y_s-\bar{Y}_s\right|^p\right] \int_{T-h}^T e^{p \beta A_s}  d A_s +e^{p \beta A_\tau}\left\{2^{p-1} \gamma_1^p\left|Y_\tau-\bar{Y}_\tau\right|^p+2^{p-1} \gamma_2^p \mathbb{E}\left[\left|Y_s-\bar{Y}_s\right|^p\right]_{\mid s=\tau}\right\} \right].
\end{aligned}
$$
Therefore,
$$
e^{p \beta A_t}\left|\widehat{\Phi}(Y)_t-\widehat{\Phi}(\bar{Y})_t\right|^p \leq \underset{\tau \in \mathcal{T}_t}{\operatorname{ess} \sup } \ \mathbb{E}_t\left[G(\tau)\right]:=V_t,
$$
in which
$$
\begin{aligned}
& G(\tau)\\
& :=2^{\frac{p}{2}-1} \left(\eta^{\frac{p}{2}} C_f^p \left(A_T-A_{T-h}\right)^{\frac{p-2}{p}} \sup_{s \in [T-h ,T]}\mathbb{E}\left[\left|Y_s-\bar{Y}_s\right|^p\right] \int_{T-h}^T e^{p\beta A_s}  d A_s +e^{p \beta A_\tau}\left(2^{p-1} \gamma_1^p\left|Y_\tau-\bar{Y}_\tau\right|^p+2^{p-1} \gamma_2^p \mathbb{E}\left[\left|Y_s-\bar{Y}_s\right|^p\right]_{\mid s=\tau}\right)\right),
\end{aligned}
$$
which yields
$$
\sup _{\tau \in \mathcal{T}_{T-h}} \mathbb{E}\left[e^{p \beta A_\tau}\left|\widehat{\Phi}(Y)_\tau-\widehat{\Phi}(\bar{Y})_\tau\right|^p\right] \leq \sup_{\tau \in \mathcal{T}_{T-h}} \mathbb{E}\left[V_\tau\right].
$$
By Lemma D.1 in \cite{karatzas1998methods}, for any $\sigma \in \mathcal{T}_{T-\delta}$, there exists a sequence $\left(\tau_n\right)_n$ of stopping times in $\mathcal{T}_\sigma$ such that
$$
V_\sigma=\lim _{n \rightarrow \infty} \mathbb{E}\left[G\left(\tau_n\right) \mid \mathcal{F}_\sigma\right],
$$
and so, by Fatou's Lemma, we have
$$
\mathbb{E}\left[V_\sigma\right] \leq \varliminf_{n \rightarrow \infty} \mathbb{E}\left[G\left(\tau_n\right)\right] \leq \sup _{\tau \in \mathcal{T}_{T-h}} \mathbb{E}[G(\tau)].
$$
Therefore,
$$
\begin{aligned}
&\sup _{\tau \in \mathcal{T}_{T-h}} \mathbb{E}\left[e^{p \beta A_\tau}\left|\widehat{\Phi}(Y)_\tau-\widehat{\Phi}(\bar{Y})_\tau\right|^p \right]\\
&\leq \sup_{\tau \in \mathcal{T}_{T-h}} \mathbb{E}[G(\tau)]\\
&= \sup_{\tau \in \mathcal{T}_{T-h}} 2^{\frac{p}{2}-1} \mathbb{E} \left[ \eta^{\frac{p}{2}} C_f^p \left(A_T-A_{T-h}\right)^{\frac{p-2}{p}} \sup_{s \in [T-h ,T]}\mathbb{E}\left[\left|Y_s-\bar{Y}_s\right|^p\right] \int_{T-h}^T e^{p \beta A_s}  d A_s \right.\\
&\quad \left.+e^{p \beta A_\tau}\left(2^{p-1} \gamma_1^p\left|Y_\tau-\bar{Y}_\tau\right|^p+2^{p-1} \gamma_2^p \mathbb{E}\left[\left|Y_s-\bar{Y}_s\right|^p\right]_{\mid s=\tau}\right) \right]\\
&\leq  2^{\frac{p}{2}-1} \eta^{\frac{p}{2}} C_f^p \left(A_T-A_{T-h}\right)^{\frac{p-2}{p}} \int_{T-h}^T e^{p \beta A_s}  d A_s \sup_{\tau \in \mathcal{T}_{T-h}}\mathbb{E}\left[\left|Y_\tau-\bar{Y}_\tau \right|^p\right] + 2^{\frac{p}{2}-1} 2^{p-1} \gamma_1^p \sup_{\tau \in \mathcal{T}_{T-h}}\mathbb{E}\left[e^{p \beta A_\tau}\left|Y_\tau-\bar{Y}_\tau\right|^p\right] \\
&\quad +2^{\frac{p}{2}-1} 2^{p-1} \gamma_2^p \mathbb{E}\left[ e^{p \beta A_T}\right]\sup_{\tau \in \mathcal{T}_{T-h}}\mathbb{E}\left[e^{p \beta A_\tau}\left|Y_\tau-\bar{Y}_\tau\right|^p\right]\\
&= \alpha \sup_{\tau \in \mathcal{T}_{T-h}}\mathbb{E}\left[e^{p \beta A_\tau}\left|Y_\tau-\bar{Y}_\tau\right|^p\right],
\end{aligned}
$$
where $\alpha:=  2^{\frac{p}{2}-1} \eta^{\frac{p}{2}} C_f^p \left(A_T-A_{T-h}\right)^{\frac{p-2}{p}} \int_{T-h}^T e^{p \beta A_s}  d A_s+  2^{\frac{p}{2}-1} 2^{p-1} \left( \gamma_1^p+\gamma_2^p \mathbb{E}\left[ e^{p \beta A_T}\right]\right)$. 
Assuming $\left(\gamma_1, \gamma_2\right)$ satisfies
$$
\gamma_1^p+\gamma_2^p \mathbb{E}\left[ e^{p \beta A_T}\right]<2^{2-\frac{3p}{2}},
$$
we can choose a small enough $h$ such that: $T=nh$ and 
$$
\|A_{ih}-A_{(i-1)h}\|_\infty^{\frac{p-2}{p}} \max_{1\le i\le n} \int_{(i-1)h}^{ih} e^{p \beta A_s}  d A_s < \frac{1}{2^{\frac{p}{2}-1} \eta^{\frac{p}{2}} C_f^p}\left( 1 - 2^{\frac{3p}{2}-2} \left( \gamma_1^p+\gamma_2^p \mathbb{E}\left[ e^{p \beta A_T}\right]\right)\right);
$$

\begin{equation}
\label{condition}
e^{p \beta A_T} \|A_{ih}-A_{(i-1)h}\|_\infty^{2-\frac{2}{p}} < \frac{1}{2^{\frac{p}{2}-1} \eta^{\frac{p}{2}} C_f^p}\left( 1 - 2^{\frac{3p}{2}-2} \left( \gamma_1^p+\gamma_2^p \mathbb{E}\left[ e^{p \beta A_T}\right]\right)\right)
\end{equation}
to make $\widehat{\Phi}$ a contraction on $\mathbb{L}_\beta^p$ over the time interval $[T-h, T]$, i.e. $\widehat{\Phi}$ admits a unique fixed point over $[T-h, T]$.

\textbf{Step 3.} We finally show that the mean-field RBSDE with jumps (\ref{reflected BSDE}) has a unique solution.
Given $Y \in \mathbb{L}_\beta^p$ the solution over $[T-h, T]$ obtained in Step 2 , let $(\hat{Y}, \hat{U}, \hat{K}) \in \mathcal{E} \times \mathcal{H}_\nu^p \times \mathbb{K}^p$ be the unique solution of the standard reflected BSDE, with barrier $h\left(s, Y_s, \mathbb{P}_{Y_s}\right)$  (which belongs to $\mathcal{E}$ by Assumption (H4') and driver $g(s, y, u):=f\left(s, y, u, \mathbb{P}_{Y_s}\right)$. Then,
$$
\hat{Y}_t=\underset{\tau \in \mathcal{T}_t}{\operatorname{ess} \sup }\  \mathcal{E}_{t, \tau}^g\left[h\left(\tau, Y_\tau,\left.\mathbb{P}_{Y_s}\right|_{s=\tau}\right) \mathbf{1}_{\{\tau<T\}}+\xi \mathbf{1}_{\{\tau=T\}}\right].
$$
By the fixed point argument from Step 2, we have $\hat{Y}=Y$ a.s., which implies that $Y \in \mathcal{E}$ and
$$
Y_t=\xi+\int_t^T f\left(s, Y_s, \hat{U}_s, \mathbb{P}_{Y_s}\right) d A_s+\hat{K}_T-\hat{K}_t-\int_t^T\int_E \hat{U}_s(e)q(ds, de) , \quad T-h\leq t \leq T,
$$
which yields existence of a solution. Furthermore, the process $Y_t$ is a nonlinear g-supermartingale. Hence, by the uniqueness of the nonlinear Doob-Meyer decomposition, we obtain uniqueness of the associated processes $(\hat{U}, \hat{K})$. This yields existence and uniqueness of (\ref{reflected BSDE}) in this general case.
Next we stitch the local solutions to get the global solution on $[0,T]$.
Choose  $h$ such that (\ref{condition}) holds and $T=n h$ for some integer $n$. Then BSDE (\ref{reflected BSDE}) admits a unique deterministic flat solution $\left(Y^n, U^n, K^n\right)$ on the time interval $[T-h, T]$. Next we take $T-h$ as the terminal time and $Y_{T-h}^n$ as the terminal condition. We can find the unique deterministic flat solution of (\ref{reflected BSDE}) $\left(Y^{n-1}, U^{n-1}, K^{n-1}\right)$  on the time interval $[T-2 h, T-h]$. 
Repeating this procedure, we get a sequence $\left(Y^i, U^i, Z^i, K^i\right)_{i \leq n}$. 
It turns out that on $[(i-1)h,ih]$, $\alpha^{(i)}:=2^{\frac{p}{2}-1} \eta^{\frac{p}{2}} C_f^p \left(A_{ih}-A_{(i-1)h}\right)^{\frac{p-2}{p}} \int_{(i-1)h}^{ih} e^{p\beta A_s}  d A_s+ 2^{\frac{p}{2}-1} 2^{p-1} \left(\gamma_1^p+\gamma_2^p \mathbb{E}\left[ e^{p\beta A_T}\right]\right) < 1$ which implies the local well-posedness on $[(i-1)h,ih]$.
We stitch the sequence as:
$$
Y_t=\sum_{i=1}^n Y_t^i I_{[(i-1) h, i h)}(t)+Y_T^n I_{\{T\}}(t), \ U_t=\sum_{i=1}^n U_t^i I_{[(i-1) h, i h)}(t)+U_T^n I_{\{T\}}(t), 
$$
as well as
$$
K_t=K_t^i+\sum_{j=1}^{i-1} K_{j h}^j, \quad \text { for } t \in[(i-1) h, i h],  i \leq n.
$$
 It is  obvious that $(Y, U, K) \in \mathcal{E} \times \mathcal{H}_\nu^p \times \mathbb{K}^p$ is a deterministic flat solution to the mean-field reflected BSDE (\ref{reflected BSDE}).
The uniqueness of the global solution follows from the uniqueness of local solution on each small time interval. The proof is complete.

Applying the same reasoning on each time interval $[T-(i+1) \delta, T-i \delta]$ with a similar dynamics but terminal condition $Y_{T-i \delta}$ at time $T-i \delta$, we build recursively for $i=1$ to any $n$ a solution $(Y, U, K) \in$ $\mathcal{E} \times H_\nu^{2,p} \times \mathbb{K}^p$ on each time interval $[T-(i+1) \delta, T-i \delta]$. Pasting these processes, we derive a unique solution $(Y, U, K)$ satisfying (\ref{reflected BSDE}) on the full time interval $[0, T]$.
\end{proof}

\section{Unbounded terminal condition  and obstacle}
\label{section exp}
In this section, inspired by Hu, Moreau and Wang\cite{hu2022quadratic} we will deal with quadratic mean-field reflected BSDE (\ref{reflected BSDE}) with unbounded terminal condition and obstacle by $\theta$-method. We first state the main result of this section.
\begin{thm}
\label{ubdd thm}
Assume that (H1)-(H4) hold. If $\gamma_1$ and $\gamma_2$ satisfy
\begin{equation}
\label{quadratic condition}
4\left(\gamma_1+\gamma_2\right)<1,
\end{equation}

then the quadratic mean-field reflected BSDE (\ref{reflected BSDE}) admits a unique solution $(Y, U, K) \in \mathcal{E}\times {H}_{\nu}^{2,p} \times \mathbb{K}^p$.
\end{thm}

In order to prove Theorem 4.1, we first recall some results about the solutions of quadratic RBSDEs driven by a marked point process. 
\hspace*{\fill}\\

\noindent(\textbf{H3''})

\textbf{(a) (Continuity condition)} 
For every $\omega \in \Omega,\ t \in[0, T],\ y \in \mathbb{R}$, $u\in L^2(\mathcal{B}(E),\nu_t;\mathbb R)$, $(y, u) \longrightarrow f(t, y, u)$ is continuous.
\hspace*{\fill}\\
\hspace*{\fill}\\

\textbf{(b) \textcolor{red}{(Lipschitz condition in $y$)}} 
There exist $\tilde{\beta}\geq 0$, such that for every $\omega \in \Omega,\ t \in[0, T],\ y, y^{\prime} \in \mathbb{R}$, $u\in L^2(\mathcal{B}(E),\nu_t;\mathbb R)$, we have
$$
\begin{aligned}
& \left|f(\omega, t, y, u(\cdot))-f\left(\omega, t, y^{\prime}, u(\cdot)\right)\right| \leq \tilde{\beta}\left|y-y^{\prime}\right|.
\end{aligned}
$$
\hspace*{\fill}\\
\hspace*{\fill}\\

\textbf{(c) \textcolor{red}{(Quadratic-exponential growth condition)}}
For all $t\in[0, T]$, $(y,u) \in \mathbb{R} \times  L^2(\mathcal{B}(E),\nu_t;\mathbb R):\ \mathbb{P}$-a.s.,
$$
\underline{q}(t, y, u)=-\frac{1}{\lambda} j_{\lambda}(t,- u)-\alpha_t-\beta|y| \leq f(t, y, u) \leq \frac{1}{\lambda} j_{\lambda}(t, u)+\alpha_t+\beta|y|=\bar{q}(t, y, u) .
$$
where $\{\alpha_t\}_{0 \leq t \leq T}$ is  a progressively measurable nonnegative stochastic process.
\hspace*{\fill}\\
\hspace*{\fill}\\

\textbf{(d) Integrability condition} We assume necessarily,
$$
\forall p>0, \quad \mathbb{E}\left[\exp \left\{p \lambda e^{\beta A_T}(|\xi|\vee L_*)+p\lambda  \int_0^T e^{\beta A_s}\alpha_s d A_s\right\}+\int_0^T\alpha_s^2dA_s\right]<\infty.
$$
\hspace*{\fill}\\

\textbf{(e) (Convexity/Concavity condition)}
For each $(t, y) \in[0, T] \times \mathbb{R}$, $u\in L^2(\mathcal{B}(E),\nu_t;\mathbb R),\ u \rightarrow f(t, y, u)$ is convex or concave.
\hspace*{\fill}\\

\begin{remark}
\label{rmk solutions}
Suppose that assumptions (H1), (H2) and (H3'') hold, then the quadratic BSDE (\ref{BSDE}) 
\begin{equation}
\label{BSDE}
y_t^\tau= \eta + \int_t^\tau f\left(s, y_s^\tau , u_s^\tau\right)dA_s-\int_t^\tau \int_E u_s(e) q(d s,d e)
\end{equation}
admits a unique solution $(y^\tau, u^\tau) \in \mathcal E\times H_\nu^{2,p}$ by Theorem 3.1 in \cite{2310.14728}. In particular, the comparison principle also holds by Theorem 3.1 in \cite{2023arXiv231020361L} (choose obstacle as $-\infty$). Besides, we also have the conclusion that under assumptions (H1), (H2) and (H3''), the quadratic reflected BSDE (\ref{reflected BSDE normal}) 
\begin{equation}
\label{reflected BSDE normal}
\left\{\begin{array}{l}
Y_t=\eta+\int_t^T f\left(s,Y_s, U_s\right) d A_s +\int_t^T d K_s-\int_t^T \int_E U_s(e) q(d s, d e), \quad 0 \leq t \leq T, \\
Y_t \geq L_t, \quad 0 \leq t \leq T, \\
\int_0^T\left(Y_{s^{-}}-L_{s^{-}}\right) d K_s=0.
\end{array}\right.
\end{equation}
admits a unique solution $(Y, U, K) \in \mathcal E\times H_{\nu}^{2,p}\times \mathbb K^p$ by Theorems 2.1 in \cite{2023arXiv231020361L}.
\end{remark} 

There are several useful a priori estimates which can be found from  lemma3.3 in \cite{2310.14728}. For the convenience of readers, we restate this lemma here.

\begin{lemma}
\label{priori estimates}
 Under assumptions (H1) and (H2), assume that $(y^\tau, u^\tau)$ is a solution to BSDE (\ref{BSDE}). Suppose that for some $p\ge 1$,
\begin{equation}
\label{integrability_condition}
\mathbb{E}\left[\exp \left\{pe^{\beta A_T}\lambda|\xi|+ p\lambda \int_0^Te^{\beta A_t} \alpha_t d t\right\}\right]<\infty.
\end{equation}
Then, we have the following a priori estimates.

(i) If $-\alpha_t-\beta|y_t|-\frac{1}{\lambda}j_\lambda(t,-u_t)\le f(t,y, u) \leq \alpha_t+\beta|y_t|+\frac{1}{\lambda}j_\lambda(t,u_t)$, then for each $t \in[0, T]$,
$$
\exp \left\{p \lambda\left|y_t^\tau\right|\right\} \leq \mathbb{E}_t\left[\exp \left\{p \lambda e^{\beta A_T}|\xi|+p\lambda  \int_t^T e^{\beta A_s}\alpha_s d A_s\right\}\right].
$$
Hence, with the help of Doob's inequality,  for each $p>0$, $E[e^{py_*}]$ is uniformly bounded. The bound is denoted by $\Xi(p)$.

(ii) If $f(t,y,u) \leq \alpha_t+\beta |y_t|+\frac{1}{\lambda}j_{\lambda}(t,u_t)$, then for each $t \in[0, T]$,
$$
\exp \left\{p \lambda\left(y_t^\tau\right)^{+}\right\} \leq \mathbb{E}_t\left[\exp \left\{p\lambda e^{\beta A_T} \xi^{+}+p\lambda  \int_t^T e^{\beta A_s}\alpha_s d A_s\right\}\right].
$$
\end{lemma}

\begin{remark}
As in \cite{2310.14728}, Corollary 3.4, if the generator satisfies $f(t,y,u)\equiv f(t,u)$, $-\alpha_t-\frac{1}{\lambda}j_\lambda(t,-u_t)\le f(t, u) \leq \alpha_t+\frac{1}{\lambda}j_\lambda(t,u_t)$ and there is a constant $p \geq 1$ such that,
\begin{equation}
\mathbb{E}\left[\exp \left\{p\lambda|\xi|+ p\lambda \int_0^T \alpha_t d t\right\}\right]<\infty.
\end{equation}
Then, we have
$$
\exp \left\{p \lambda\left|y_t\right|\right\} \leq \mathbb{E}_t\left[\exp \left\{p \lambda |\xi|+p\lambda  \int_t^T \alpha_s d A_s\right\}\right].
$$
\end{remark}

Now we are going to prove the uniqueness of the solution of the quadratic mean-field reflected BSDE (\ref{reflected BSDE}).

\begin{lemma}
Assume that all the conditions of Theorem \ref{ubdd thm} hold. Then, the quadratic mean-field reflected BSDE (\ref{reflected BSDE}) has at most one solution $(Y, U, K) \in \mathcal{E}\times {H}_{\nu}^{2,p} \times \mathbb{K}^p$.
\end{lemma}

\begin{proof}
For $i=1,2$, let $\left(Y^i, U^i, K^i\right) \in \mathcal{E} \times H_\nu^{2,p} \times \mathbb{K}^p$ are solutions to the quadratic mean-field reflected BSDE (\ref{reflected BSDE}). From Lemma \ref{representation} and Remark \ref{rmk solutions}, we have
$$
Y_t^i:=\underset{\tau \in \mathcal{T}_t}{\operatorname{ess} \sup }\ y_t^{i, \tau}, \quad \forall t \in[0, T],
$$
where $y_t^{i, \tau}$ is the solution of the following quadratic BSDE
$$
y_t^{i, \tau}=\xi \mathbf{1}_{\{\tau=T\}}+h\left(\tau, Y_\tau^i,\left(\mathbb{P}_{Y_s^i}\right)_{s=\tau}\right) \mathbf{1}_{\{\tau<T\}}+\int_t^\tau f\left(s, Y_s^i, \mathbb{P}_{Y_s^i}, u_s^{i, \tau}\right) d A_s-\int_t^\tau \int_E u_s^{i, \tau} q(ds, de).
$$
Without loss of generality, assume that $f(t, y, \mu, \cdot)$ is concave, for each $\theta \in(0,1)$, denote by
$$
\delta_\theta \ell=\frac{ \ell^1- \theta\ell^2}{1-\theta}, \delta_\theta \widetilde{\ell}=\frac{ \ell^2-\theta \ell^1}{1-\theta} \text { and } \delta_\theta \bar{Y}:=\left|\delta_\theta Y\right|+\left|\delta_\theta \widetilde{Y}\right|
$$
for $\ell=Y, y^\tau$ and $u^\tau$. Then, the pair of processes $\left(\delta_\theta y^\tau, \delta_\theta u^\tau\right)$ satisfies the following BSDE:
\begin{equation}
\label{delta BSDE}
\delta_\theta y_t^\tau=\delta_\theta \eta+\int_t^\tau\left(\delta_\theta f\left(s, \delta_\theta u_s^\tau\right)+\delta_\theta f_0(s)\right) d A_s-\int_t^\tau  \int_E \delta_\theta u_s^\tau q(ds,de),
\end{equation}
where the terminal condition and generator are given by
$$
\begin{aligned}
& \delta_\theta \eta=\xi \mathbf{1}_{\{\tau=T\}}+\frac{ h\left(\tau, Y_\tau^1,\left(\mathbb{P}_{Y_s^1}\right)_{s=\tau}\right)- \theta h\left(\tau, Y_\tau^2,\left(\mathbb{P}_{Y_s^2}\right)_{s=\tau}\right)}{1-\theta} \mathbf{1}_{\{\tau<T\}}, \\
& \delta_\theta f_0(t)=\frac{1}{1-\theta}\left(f\left(t, Y_t^1, \mathbb{P}_{Y_t^1}, u_t^{1, \tau}\right)-f\left(t, Y_t^2, \mathbb{P}_{Y_t^2}, u_t^{1, \tau}\right)\right), \\
& \delta_\theta f(t, u)=\frac{1}{1-\theta}\left( f\left(t, Y_t^2, \mathbb{P}_{Y_t^2},(1-\theta) u+\theta u_t^{2, \tau}\right)-\theta f\left(t, Y_t^2, \mathbb{P}_{Y_t^2}, u_t^{2, \tau}\right)\right) .
\end{aligned}
$$
Recalling assumptions (H3) and (H4), we have that
$$
\begin{aligned}
& \left.\left.\delta_\theta \eta \leq|\xi|+\left|h\left(\tau, 0, \delta_0\right)\right|+\gamma_1\left(2\left|Y_\tau^2\right|+\left|\delta_\theta Y_\tau\right|\right)+\gamma_2\left(2 \mathbb{E}\left[\left|Y_s^2\right|\right]\right]_{s=\tau}+\mathbb{E}\left[\left|\delta_\theta Y_s\right|\right]_{s=\tau}\right)\right), \\
& \delta_\theta f_0(t) \leq \beta\left(\left|Y_t^2\right|+\left|\delta_\theta Y_t\right|+\mathbb{E}\left[\left|Y_t^2\right|\right]+\mathbb{E}\left[\left|\delta_\theta Y_t\right|\right]\right), \\
& \delta_\theta f(t, u) \leq f\left(t, Y_t^2, \mathbb{P}_{Y_t^2},u\right) \leq \alpha_t+\beta\left(\left|Y_t^2\right|+\mathbb{E}\left[\left|Y_t^2\right|\right]\right)+\frac{1}{\lambda}j_\lambda(t, u).
\end{aligned}
$$
Set $C_1:=\sup _{i \in\{1,2\}} \mathbb{E}\left[\sup _{s \in[0, T]}\left|Y_s^i\right|\right]$ and
$$
\begin{aligned}
& \chi=\int_0^T \alpha_s d A_s+\sup _{s \in[0, T]}\left|h\left(s, 0, \delta_0\right)\right|+2\left(\gamma_1+\beta A_T\right)\left(\sup _{s \in[0, T]}\left|Y_s^1\right|+\sup _{s \in[0, T]}\left|Y_s^2\right|\right)+2\left(\gamma_2+\beta A_T\right) C_1, \\
& \widetilde{\chi}=\int_0^T \alpha_s d A_s+\sup _{s \in[0, T]}\left|h\left(s, 0, \delta_0\right)\right|+2\left(1+\gamma_1+\beta A_T\right)\left(\sup _{s \in[0, T]}\left|Y_s^1\right|+\sup _{s \in[0, T]}\left|Y_s^2\right|\right)+2\left(\gamma_2+\beta A_T\right) C_1.
\end{aligned}
$$
Using assertion (ii) of Lemma \ref{priori estimates} to (\ref{delta BSDE}), we derive that for any $p \geq 1$
$$
\begin{aligned}
& \exp \left\{p \gamma\left(\delta_\theta y_t^\tau\right)^{+}\right\} \\
& \leq \mathbb{E}_t\left[\exp \left\{p \gamma\left(|\xi|+\chi+\left(\gamma_1+\beta(A_T-A_t)\right) \sup _{s \in[t, T]}\left|\delta_\theta Y_s\right|+\left(\gamma_2+\beta(A_T-A_t)\right) \sup _{s \in[t, T]} \mathbb{E}\left[\left|\delta_\theta Y_s\right|\right]\right)\right\}\right],
\end{aligned}
$$
which indicates that
\begin{equation}
\begin{aligned}
\label{estimate Y}
& \exp \left\{p \gamma\left(\delta_\theta Y_t\right)^{+}\right\} \leq \underset{\tau \in \mathcal{T}_t}{\operatorname{ess} \sup \exp }\left\{p \gamma\left(\delta_\theta y_t^\tau\right)^{+}\right\} \\
& \leq \mathbb{E}_t\left[\exp \left\{p \gamma\left(|\xi|+\chi+\left(\gamma_1+\beta(A_T-A_t)\right) \sup _{s \in[t, T]}\left|\delta_\theta Y_s\right|+\left(\gamma_2+\beta(A_T-A_t)\right) \sup _{s \in[t, T]} \mathbb{E}\left[\left|\delta_\theta Y_s\right|\right]\right)\right\}\right].
\end{aligned}  
\end{equation}

Using a similar method, we derive that
\begin{equation}
\begin{aligned}
\label{estimate Y tilde}
& \exp \left\{p \gamma\left(\delta_\theta \widetilde{Y}_t\right)^{+}\right\} \\
& \leq \mathbb{E}_t\left[\exp \left\{p \gamma\left(|\xi|+\chi+\left(\gamma_1+\beta(A_T-A_t)\right) \sup _{s \in[t, T]}\left|\delta_\theta \widetilde{Y}_s\right|+\left(\gamma_2+\beta(A_T-A_t)\right) \sup _{s \in[t, T]} \mathbb{E}\left[\left|\delta_\theta \widetilde{Y}_s\right|\right]\right)\right\}\right] .
\end{aligned} 
\end{equation}

In view of the fact that
$$
\left(\delta_\theta Y\right)^{-} \leq\left(\delta_\theta \widetilde{Y}\right)^{+}+2\left|Y^2\right| \text { and }\left(\delta_\theta \widetilde{Y}\right)^{-} \leq\left(\delta_\theta Y\right)^{+}+2\left|Y^1\right|,
$$
and recalling (\ref{estimate Y}) and (\ref{estimate Y tilde}), we have
$$
\begin{aligned}
& \exp \left\{p \gamma\left|\delta_\theta Y_t\right|\right\} \vee \exp \left\{p \gamma\left|\delta_\theta \widetilde{Y}_t\right|\right\} \leq \exp \left\{p \gamma\left(\left(\delta_\theta Y_t\right)^{+}+\left(\delta_\theta \widetilde{Y}_t\right)^{+}+2\left|Y_t^1\right|+2\left|Y_t^2\right|\right)\right\} \\
& \leq \mathbb{E}_t\left[\exp \left\{p \gamma\left(|\xi|+\widetilde{\chi}+\left(\gamma_1+\beta(A_T-A_t)\right) \sup _{s \in[t, T]} \delta_\theta \bar{Y}_s+\left(\gamma_2+\beta(A_T-A_t)\right) \sup _{s \in[t, T]} \mathbb{E}\left[\delta_\theta \bar{Y}_s\right]\right)\right\}\right]^2.
\end{aligned}
$$
Applying Doob's maximal inequality and Hölder's inequality, we get that for each $p \geq 1$ and $t \in[0, T]$
$$
\begin{aligned}
& \mathbb{E}\left[\exp \left\{p \gamma \sup _{s \in[t, T]} \delta_\theta \bar{Y}_s\right\}\right] \leq \mathbb{E}\left[\exp \left\{p \gamma \sup _{s \in[t, T]}\left|\delta_\theta Y_s\right|\right\} \exp \left\{p \gamma \sup _{s \in[t, T]}\left|\delta_\theta \tilde{Y}_s\right|\right\}\right] \\
& \leq 4 \mathbb{E}\left[\exp \left\{4 p \gamma\left(|\xi|+\tilde{\chi}+\left(\gamma_1+\beta(A_T-A_t)\right) \sup _{s \in[t, T]} \delta_\theta \bar{Y}_s+\left(\gamma_2+\beta(A_T-A_t)\right) \sup _{s \in[t, T]} \mathbb{E}\left[\delta_\theta \bar{Y}_s\right]\right)\right\}\right] \\
& \leq 4 \mathbb{E}\left[\exp \left\{4 p \gamma\left(|\xi|+\tilde{\chi}+\left(\gamma_1+\beta(A_T-A_t)\right) \sup _{s \in[t, T]} \delta_\theta \bar{Y}_s\right)\right\}\right] \mathbb{E}\left[\exp \left\{4 p \gamma\left(\gamma_2+\beta(A_T-A_t)\right) \sup _{s \in[t, T]} \delta_\theta \bar{Y}_s\right\}\right],
\end{aligned}
$$
where we used Jensen's inequality in the last inequality.

In view of $4\left(\gamma_1+\gamma_2\right)<1$, there exist two constants $h \in(0, T]$ and $\nu>1$ depending only on $\beta, \gamma_1$ and $\gamma_2$ such that
$$
\max_{1\le i\le N}\{4\left(\gamma_1+\gamma_2+2 \beta \|A_{ih}-A_{(i-1)h}\|_\infty\right)\}<1 \text { and } \max_{1\le i\le N}\{4 \nu\left(\gamma_1+\beta \|A_{ih}-A_{(i-1)h}\|_\infty\right)\}<1,
$$
in which $T=Nh$.
In the spirit of Hölder's inequality, we derive that for any $p \geq 1$
$$
\begin{aligned}
& \mathbb{E}\left[\exp \left\{p \gamma \sup _{s \in[T-h, T]} \delta_\theta \bar{Y}_s\right\}\right] \\
& \leq 4 \mathbb{E}\left[\exp \left\{\frac{4 \nu p \gamma}{\nu-1}(|\xi|+\tilde{\chi})\right\}\right]^{\frac{\nu-1}{\nu}} \mathbb{E}\left[\exp \left\{4 \nu p \gamma\left(\gamma_1+\beta \|A_{ih}-A_{(i-1)h}\|_\infty\right) \sup _{s \in[t, T]} \delta_\theta \bar{Y}_s\right\}\right]^{\frac{1}{\nu}} \\
& \quad \times \mathbb{E}\left[\exp \left\{p \gamma \sup _{s \in[t, T]} \delta_\theta \bar{Y}_s\right\}\right]^{4\left(\gamma_2+\beta \|A_{ih}-A_{(i-1)h}\|_\infty\right)} \\
& \leq 4 \mathbb{E}\left[\exp \left\{\frac{4 \nu p \gamma}{\nu-1}(|\xi|+\tilde{\chi})\right\}\right]^{\frac{\nu-1}{\nu}} \mathbb{E}\left[\exp \left\{p \gamma \sup _{s \in[T-h, T]} \delta_\theta \bar{Y}_s\right\}\right]^{4\left(\gamma_1+\gamma_2+2 \beta \|A_{ih}-A_{(i-1)h}\|_\infty\right)},
\end{aligned}
$$
which together with the fact that $4\left(\gamma_1+\gamma_2+2 \beta \|A_{ih}-A_{(i-1)h}\|_\infty\right)<1$ implies that for any $p \geq 1$ and $\theta \in(0,1)$,
$$
\mathbb{E}\left[\exp \left\{p \gamma \sup _{s \in[T-h, T]} \delta_\theta \bar{Y}_s\right\}\right] \leq 4^{\frac{1}{1-4\left(\gamma_1+\gamma_2+2 \beta \|A_{ih}-A_{(i-1)h}\|_\infty\right)}} \mathbb{E}\left[\exp \left\{\frac{4 \nu p \gamma}{\nu-1}(|\xi|+\widetilde{\chi})\right\}\right]^{\frac{\nu-1}{\nu\left(1-4\left(\gamma_1+\gamma_2+2 \beta \|A_{ih}-A_{(i-1)h}\|_\infty\right)\right.}}<\infty.
$$
Note that $Y^1-Y^2=(1-\theta)\left(\delta_\theta Y+Y^1\right)$. It follows that
$$
\mathbb{E}\left[\sup _{t \in[T-h, T]}\left|Y_t^1-Y_t^2\right|\right] \leq(1-\theta)\left(\frac{1}{\gamma} \mathbb{E}\left[4^{\frac{\nu}{\nu-1}} \exp \left\{\frac{4 \nu \gamma}{\nu-1}(|\xi|+\widetilde{\chi})\right\}\right]^{\frac{\nu-1}{\nu\left(1-4\left(\gamma_1+\gamma_2+2 \beta \|A_{ih}-A_{(i-1)h}\|_\infty\right)\right)}}+\mathbb{E}\left[\sup _{t \in[0, T]}\left|Y_t^1\right|\right]\right).
$$
Letting $\theta \rightarrow 1$ yields that $Y^1=Y^2$ and then $\left(U^1, K^1\right)=\left(U^2, K^2\right)$ on $[T-h, T]$. Repeating iteratively this procedure a finite number of times, we get the uniqueness on the given interval $[0, T]$. The proof is complete.
\end{proof}

\begin{remark}
In the concave case, we can use $\theta \ell^1- \ell^2$ and $\theta \ell^2-\ell^1$ instead of $ \ell^1-\theta\ell^2$ and $\ell^2-\theta\ell^1$ in the definition of $\delta_\theta \ell$ and $\delta_\theta \widetilde{\ell}$. The proof holds from line to line.
\end{remark}

Let us now turn to the proof of existence.
\begin{lemma}
Assume that all the conditions of Theorem \ref{ubdd thm} hold and $P \in \mathcal{E}$ with $P_T=\xi$. Then, the following quadratic reflected BSDE :
$$
\left\{\begin{array}{l}
Y_t=\xi+\int_t^T f\left(s, Y_s, U_s, \mathbb{P}_{P_s}\right) d A_s-\int_t^T \int_E U_s(e) q(ds, de) +K_T-K_t, \quad 0 \leq t \leq T \\
Y_t \geq h\left(t, P_t, \mathbb{P}_{P_t}\right), \quad \forall t \in[0, T] \text { and } \quad \int_0^T\left(Y_{t^{-}}-h\left(t, P_{t^-}, \mathbb{P}_{P_{t^-}}\right)\right) d K_t=0.
\end{array}\right.
$$
admits a unique solution $(Y, U, K) \in \mathcal{E}\times {H}_{\nu}^{2,p} \times \mathbb{K}^p$.
\end{lemma}

\begin{proof}
From Assumptions (H3) and (H4), we have
\begin{equation}
\label{fixed f and h}
\left|f\left(t, y, u, \mathbb{P}_{P_t}\right)\right| \leq \alpha_t+\beta\left(y+\mathbb{E}\left[\left|P_t\right|\right]\right)+\frac{1}{\lambda}j_\lambda(t, u), \quad\left|h\left(t, P_t, \mathbb{P}_{P_t}\right)\right| \leq\left|h\left(t, 0, \delta_0\right)\right|+\gamma_1\left|P_t\right|+\gamma_2 \mathbb{E}\left[\left|P_t\right|\right] .
\end{equation}

It follows from theorem 3.1 in \cite{2023arXiv231020361L} that the driver $(t, y, z) \mapsto f\left(t, y, u, \mathbb{P}_{P_t}\right)$ satisfies the comparison principle and  by applying theorem 2.1 in \cite{2023arXiv231020361L}, we get the desired result.
\end{proof}

Define recursively a sequence of stochastic processes $\left(Y^{(m)}\right)_{m=1}^{\infty}$ through the following quadratic reflected BSDE:
$$
\left\{\begin{array}{l}
Y_t^{(m)}=\xi+\int_t^T f\left(s, Y_s^{(m)}, U_s^{(m)}, \mathbb{P}_{Y_s^{(m-1)}} \right) d A_s-\int_t^T \int_E U_s(e)^{(m)} q(ds, de)+K_T^{(m)}-K_t^{(m)}, \quad 0 \leq t \leq T, \\
Y_t^{(m)} \geq h\left(t, Y_t^{(m-1)}, \mathbb{P}_{Y_t^{(m-1)}}\right), \quad \forall t \in[0, T] \text { and } \int_0^T\left(Y_t^{(m)}-h\left(t, Y_t^{(m-1)}, \mathbb{P}_{Y_t^{(m-1)}}\right)\right) d K_t^{(m)}=0,
\end{array}\right.
$$
where $Y_t^{(0)}=\mathbb{E}_t[\xi]$ for $t \in[0, T]$. In particular, we have that $\left(Y^{(m)}, Z^{(m)}, K^{(m)}\right) \in \mathcal E \times {H}_{\nu}^{2,p} \times \mathbb{K}^p$. 

\begin{remark}
As mentioned in \cite{hu2022quadratic} that if $\mathbb{E}\left[ \exp\{p|\xi|\} \right] < \infty$ for any $p\geq 1$, then the process $\left(\mathbb{E}_t[\xi]\right)_{0 \leq t \leq T} \in \mathcal{E}$. Thus we can recursively define the above quadratic reflected BSDE.
\end{remark}

Next, we use a $\theta$-method to prove that the limit of $Y^{(m)}$ is a desired solution. The following uniform estimates are crucial for our main result.

\begin{lemma}
\label{estimate 1}
Assume that the conditions of theorem 4.1 are fulfilled. Then, for any $p \geq 1$, we have
$$
\sup _{m \geq 0} \mathbb{E}\left[\exp \left\{p \gamma \sup _{s \in[0, T]}\left|Y_s^{(m)}\right|\right\}+\left(\int_0^T\int_E\left|U_t(e)^{(m)}\right|^2\phi_t(de)dA_t\right)^{p/2}+\left|K_T^{(m)}\right|^p\right]<\infty .
$$
\end{lemma}

\begin{proof}
The proof will be given in Appendix.
\end{proof}

\begin{lemma}
\label{estimate 2}
Assume that all the conditions of Theorem \ref{ubdd thm} hold. Then, for any $p \geq 1$, we have
$$
\Pi(p):=\sup _{\theta \in(0,1)} \lim _{m \rightarrow \infty} \sup _{q \geq 1} \mathbb{E}\left[\exp \left\{p \gamma \sup _{s \in[0, T]} \delta_\theta \bar{Y}_s^{(m, q)}\right\}\right]<\infty,
$$
where we use the following notations
$$
\delta_\theta Y^{(m, q)}=\frac{ Y^{(m+q)}-\theta Y^{(m)}}{1-\theta}, \delta_\theta \widetilde{Y}^{(m, q)}=\frac{ Y^{(m)}-\theta Y^{(m+q)}}{1-\theta} \text { and } \delta_\theta \bar{Y}:=\left|\delta_\theta Y^{(m, q)}\right|+\left|\delta_\theta \widetilde{Y}^{(m, q)}\right| .
$$
\end{lemma}

\begin{proof}
The proof will be given in Appendix.
\end{proof}
We are now in a position to complete the proof of the main result.

\begin{proof}[Proof of Theorem \ref{ubdd thm}]
It is enough to prove the existence. Note that for any integer $p \geq 1$ and $\theta \in(0,1)$,
$$
\limsup _{m \rightarrow \infty} \sup _{q \geq 1} \mathbb{E}\left[\sup _{t \in[0, T]}\left|Y_t^{(m+q)}-Y_t^{(m)}\right|^p\right] \leq 2^{p-1}(1-\theta)^p\left(\frac{\Pi(1) p !}{\gamma^p}+\sup _{m \geq 1} \mathbb{E}\left[\sup _{t \in[0, T]}\left|Y_t^{(m)}\right|^p\right]\right) .
$$
Sending $\theta \rightarrow 1$ and recalling Lemmas \ref{estimate 1} and \ref{estimate 2} , we could find a process $Y \in \mathcal E$ such that
\begin{equation}
\label{eq limit}
\lim _{m \rightarrow \infty} \mathbb{E}\left[\sup _{t \in[0, T]}\left|Y_t^{(m)}-Y_t\right|^p\right]=0, \quad \forall p \geq 1 .
\end{equation}
Applying Itô's formula to $\left|Y_t^{(m+q)}-Y_t^{(m)}\right|^2$ and by a standard calculus, we have
$$
\begin{aligned}
& \mathbb{E}\left[\int_0^T\left|U_t(e)^{(m+q)}-U_t(e)^{(m)}\right|^2 \phi_s(de)d A_t\right] \leq 2\mathbb{E}\left[\sup _{t \in[0, T]}\left|Y_t^{(m+q)}-Y_t^{(m)}\right|^2+\sup _{t \in[0, T]}\left|Y_t^{(m+q)}-Y_t^{(m)}\right| \Delta^{(m, q)}\right] \\
& \leq 2\mathbb{E}\left[\sup _{t \in[0, T]}\left|Y_t^{(m+q)}-Y_t^{(m)}\right|^2\right]+\mathbb{E}\left[\left|\Delta^{(m, q)}\right|^2\right]^{\frac{1}{2}} \mathbb{E}\left[\sup _{t \in[0, T]}\left|Y_t^{(m+q)}-Y_t^{(m)}\right|^2\right]^{\frac{1}{2}},
\end{aligned}
$$
with
$$
\Delta^{(m, q)}:=\int_0^T\left|f\left(t, Y_t^{(m+q)}, U_t^{(m+q)}, \mathbb{P}_{Y_t^{(m+q-1)}}\right)-f\left(t, Y_t^{(m)}, U_t^{(m)}, \mathbb{P}_{Y_t^{(m-1)}}\right)\right| d A_t+\left|K_T^{(m+q)}\right|+\left|K_T^{(m)}\right|,
$$
which together with Lemma \ref{estimate 1}, (\ref{eq limit}) and dominated convergence theorem indicates that there exists a process $U \in {H}_{\nu}^{2,p}$ so that
$$
\lim _{m \rightarrow \infty} \mathbb{E}\left[\left(\int_0^T\left|U_t^{(m)}-U_t\right|^2 \phi_s(de)d A_t\right)^p\right]=0,\quad \forall p \geq 1
$$
Set
$$
K_t=Y_t-Y_0+\int_0^t f\left(s, Y_s, U_s, \mathbb{P}_{Y_s} \right) d A_s-\int_0^t\int_E U_s(e) q(ds,de).
$$
Applying dominated convergence theorem again yields that for each $p \geq 1$
$$
\lim _{m \rightarrow \infty} \mathbb{E}\left[\left(\int_0^T\left|f\left(t, Y_t^{(m)}, U_t^{(m)}, \mathbb{P}_{Y_t^{(m-1)}}\right)-f\left(t, Y_t, U_t, \mathbb{P}_{Y_t}\right)\right| d A_t\right)^p\right]=0,
$$
which implies that $\mathbb{E}\left[\sup _{t \in[0, T]}\left|K_t-K_t^{(m)}\right|^p\right] \rightarrow 0$ as $m \rightarrow \infty$ for each $p \geq 1$ and that $K$ is a nondecreasing process. Note that
$$
\lim _{m \rightarrow \infty} \mathbb{E}\left[\sup _{t \in[0, T]}\left|h\left(t, Y_t^{(m-1)}, \mathbb{P}_{Y_t^{(m-1)}}\right)-h\left(t, Y_t, \mathbb{P}_{Y_t}\right)\right|\right] \leq\left(\gamma_1+\gamma_2\right) \lim _{m \rightarrow \infty} \mathbb{E}\left[\sup _{t \in[0, T]}\left|Y_t^{(m-1)}-Y_t\right|\right]=0.
$$
Then it is obvious that $Y_t \geq h\left(t, Y_t, \mathbb{P}_{Y_t}\right)$. Moreover, recalling lemma 13 in \cite{briand2018bsdes}, we have
$$
\int_0^T\left(Y_t-h\left(t, Y_t, \mathbb{P}_{Y_t}\right)\right) d K_t=\lim _{m \rightarrow \infty} \int_0^T\left(Y_t^{(m)}-h\left(t, Y_t^{(m-1)}, \mathbb{P}_{\left.Y_t^{(m-1)}\right)}\right) d K_t^{(m)}=0\right.,
$$
which implies that $(Y, Z, K) \in \mathcal{E}\times {H}_{\nu}^{2,p} \times \mathbb{K}^p$ is a solution to quadratic mean-field reflected (\ref{reflected BSDE}). The proof is complete.  
\end{proof}

\begin{appendices}
\section{Proofs of Lemma \ref{estimate 1} and Lemma \ref{estimate 2}}
\label{appexdix A}
We now turn to the proofs of Lemma \ref{estimate 1} and Lemma \ref{estimate 2}.
\begin{proof}[Proof of Theorem \ref{estimate 1}]
It follows from Lemma \ref{representation} and Remark \ref{rmk solutions} that for any $m \geq 1$
\begin{equation}
\label{eq representation}
Y_t^{(m)}:=\underset{\tau \in \mathcal{T}_t}{\operatorname{ess} \sup } 
\ y_t^{(m), \tau}, \quad \forall t \in[0, T],
\end{equation}
in which $y_t^{(m), \tau}$ is the solution of the following quadratic BSDE
$$
y_t^{(m), \tau}=\xi \mathbf{1}_{\{\tau=T\}}+h\left(\tau, Y_\tau^{(m-1)},\left(\mathbb{P}_{Y_{s}^{(m-1)}}\right)_{s=\tau}\right) \mathbf{1}_{\{\tau<T\}}+\int_t^\tau f\left(s, y_s^{(m), \tau}, u_s^{(m), \tau}, \mathbb{P}_{Y_s^{(m-1)}} \right) d A_s-\int_t^\tau \int_E u_s(e)^{(m), \tau} q(ds, de).
$$
Thanks to assertion (i) of Lemma \ref{priori estimates} (taking $\alpha_t=\alpha_t+\beta \mathbb{E}\left[\left|Y_t^{(m-1)}\right|\right]$ ) and in view of (\ref{fixed f and h}), we get for any $t \in[0, T]$,
\begin{equation}
\label{eq estimate y}
\begin{aligned}
& \exp \left\{\lambda\left|y_t^{(m), \tau}\right|\right\} \\
& \leq \mathbb{E}_t\left[\exp \left\{\lambda e^{\beta A_T}\left(|\xi|+\eta+\gamma_1 \sup _{s \in[t, T]}\left|Y_s^{(m-1)}\right|+\left(\gamma_2+\beta(A_T-A_t)\right) \sup _{s \in[t, T]} \mathbb{E}\left[\left|Y_s^{(m-1)}\right|\right]\right)\right\}\right],
\end{aligned}
\end{equation}
in which $\eta=\int_0^T \alpha_s d A_s+\sup _{s \in[0, T]}\left|h\left(s, 0, \delta_0\right)\right|$. Recalling (\ref{eq representation}) and applying Doob's maximal inequality and Jensen's inequality, we get that for each $m \geq 1, p \geq 2$ and $t \in[0, T]$
$$
\begin{aligned}
\mathbb{E}\left[\exp \left\{p \lambda \sup _{s \in[t, T]}\left|Y_s^{(m)}\right|\right\}\right]  &\leq 4 \mathbb{E}\left[\exp \left\{p \lambda e^{\beta A_T}\left(|\xi|+\eta+\gamma_1 \sup _{s \in[t, T]}\left|Y_s^{(m-1)}\right|\right)\right\}\right] \\
& \times \mathbb{E}\left[\exp \left\{p \lambda e^{\beta A_T}\left(\gamma_2+\beta(A_T-A_t)\right) \sup _{s \in[t, T]}\left|Y_s^{(m-1)}\right|\right\}\right].
\end{aligned}
$$
Under assumption (\ref{quadratic condition}), we can then find three constants $h \in(0, T]$ and $\nu, \widetilde{\nu}>1$ depending only on $\beta, A_T, \gamma_1$ and $\gamma_2$ such that
$$
\max_{1\le i\le N}\{4 e^{\beta A_T} \widetilde{\nu}\left(\gamma_1+\gamma_2+\beta \|A_{ih}-A_{(i-1)h}\|_\infty\right) \}<1 \text { and } 4 e^{\beta A_T} \nu \widetilde{\nu} \gamma_1<1,
$$
where $T=Nh$.

In the spirit of Hölder's inequality, we derive that for any $p \geq 2$
$$
\begin{aligned}
& \mathbb{E}\left[\exp \left\{p \lambda \sup _{s \in[T-h, T]}\left|Y_s^{(m)}\right|\right\}\right] \\
& \leq 4 \mathbb{E}\left[\exp \left\{\frac{\nu p \lambda}{\nu-1} e^{\beta A_T}(|\xi|+\eta)\right\}\right]^{\frac{\nu-1}{\nu}} \mathbb{E}\left[\exp \left\{p \lambda \sup _{s \in[T-h, T]}\left|Y_s^{(m-1)}\right|\right\}\right]^{e^{\beta A_T}\left(\gamma_1+\gamma_2+\beta \|A_{ih}-A_{(i-1)h}\|_\infty\right)} \\
& \leq 4 \mathbb{E}\left[\exp \left\{\frac{2 \nu p \lambda}{\nu-1} e^{\beta A_T}|\xi|\right\}\right]^{\frac{\nu-1}{2 \nu}} \mathbb{E}\left[\exp \left\{\frac{2 \nu p \lambda}{\nu-1} e^{\beta A_T} \eta\right\}\right]^{\frac{\nu-1}{2 \nu}} \mathbb{E}\left[\exp \left\{p \lambda \sup _{s \in[T-h, T]}\left|Y_s^{(m-1)}\right|\right\}\right]^{e^{\beta A_T}\left(\gamma_1+\gamma_2+\beta \|A_{ih}-A_{(i-1)h}\|_\infty\right)} .
\end{aligned}
$$
Define $\rho=\frac{1}{1-e^{\beta A_T}\left(\gamma_1+\gamma_2+\beta \|A_{ih}-A_{(i-1)h}\|_\infty\right)}$.

If $N=1$, it follows from the previous inequality that for each $p \geq 2$ and $m \geq 1$
$$
\begin{aligned}
& \mathbb{E}\left[\exp \left\{p \lambda \sup _{s \in[0, T]}\left|Y_s^{(m)}\right|\right\}\right] \\
& \leq 4 \mathbb{E}\left[\exp \left\{\frac{2 \nu p \lambda}{\nu-1} e^{\beta A_T}|\xi|\right\}\right]^{\frac{1}{2}} \mathbb{E}\left[\exp \left\{\frac{2 \nu p \lambda}{\nu-1} e^{\beta A_T} \eta\right\}\right]^{\frac{1}{2}} \mathbb{E}\left[\exp \left\{p \lambda \sup _{s \in[0, T]}\left|Y_s^{(m-1)}\right|\right\}\right]^{e^{\beta A_T}\left(\gamma_1+\gamma_2+\beta \|A_{T}-A_{0}\|_\infty\right)} .
\end{aligned}
$$
Iterating the above procedure $m$ times, we get,
\begin{equation}
\label{eq iteration}
\begin{aligned}
& \mathbb{E}\left[\exp \left\{p \lambda \sup _{s \in[0, T]}\left|Y_s^{(m)}\right|\right\}\right] \\
& \leq 4^\rho \mathbb{E}\left[\exp \left\{\frac{2 \nu p \lambda}{\nu-1} e^{\beta A_T}|\xi|\right\}\right]^{\frac{\rho}{2}} \mathbb{E}\left[\exp \left\{\frac{2 \nu p \lambda}{\nu-1} e^{\beta A_T} \eta\right\}\right]^{\frac{\rho}{2}} \mathbb{E}\left[\exp \left\{p \lambda \sup _{s \in[0, T]}\left|Y_s^{(0)}\right|\right\}\right]^{e^{m \beta A_T}\left(\gamma_1+\gamma_2+\beta \|A_{T}-A_{0}\|_\infty\right)^m},
\end{aligned}  
\end{equation}
which is uniformly bounded with respect to $m$. If $N=2$, proceeding identically as in the above, we have for any $p \geq 2$,
\begin{equation}
\label{eq after iteration}
\begin{aligned}
& \mathbb{E}\left[\exp \left\{p \lambda \sup _{s \in[T-h, T]}\left|Y_s^{(m)}\right|\right\}\right] \\
& \leq 4^\rho \mathbb{E}\left[\exp \left\{\frac{2 \nu p \lambda}{\nu-1} e^{\beta A_T}|\xi|\right\}\right]^{\frac{\rho}{2}} \mathbb{E}\left[\exp \left\{\frac{2 \nu p \lambda}{\nu-1} e^{\beta A_T} \eta\right\}\right]^{\frac{\rho}{2}} \mathbb{E}\left[\exp \left\{p \lambda \sup _{s \in[0, T]}\left|Y_s^{(0)}\right|\right\}\right]^{e^{m \beta A_T}\left(\gamma_1+\gamma_2+\beta \|A_{ih}-A_{(i-1)h}\|_\infty\right)^m}.
\end{aligned}
\end{equation}
Then, consider the following quadratic reflected BSDEs on time interval $[0, T-h]$ :
$$
\left\{\begin{array}{l}
Y_t^{(m)}=Y_{T-h}^{(m)}+\int_t^{T-h} f\left(s, Y_s^{(m-1)}, U_s^{(m)}, \mathbb{P}_{Y_s^{(m-1)}}, \right) d A_s-\int_t^{T-h}\int_E U_s(e)^{(m)} q(ds, de)+K_T^{(m)}-K_t^{(m)}, \quad 0 \leq t \leq T-h, \\
Y_t^{(m)} \geq h\left(t, Y_t^{(m-1)}, \mathbb{P}_{Y_t^{(m-1)}}\right), \quad \forall t \in[0, T-h] \text { and } \int_0^{T-h}\left(Y_t^{(m)}-h\left(t, Y_t^{(m-1)}, \mathbb{P}_{\left.Y_t^{(m-1)}\right)}\right)\right) d K_t^{(m)}=0 .
\end{array}\right.
$$
In view of the derivation of (\ref{eq iteration}), we deduce that
$$
\begin{aligned}
& \mathbb{E}\left[\exp \left\{p \lambda \sup _{s \in[0, T-h]}\left|Y_s^{(m)}\right|\right\}\right] \\
& \leq 4^\rho \mathbb{E}\left[\exp \left\{\frac{2 \nu p \lambda}{\nu-1} e^{\beta A_T}\left|Y_{T-h}^{(m)}\right|\right\}\right]^{\frac{\rho}{2}} \mathbb{E}\left[\exp \left\{\frac{2 \nu p \lambda}{\nu-1} e^{\beta A_T} \eta\right\}\right]^{\frac{\rho}{2}} \mathbb{E}\left[\exp \left\{p \lambda \sup _{s \in[0, T]}\left|Y_s^{(0)}\right|\right\}\right]^{e^{m \beta A_T}\left(\gamma_1+\gamma_2+\beta \|A_{ih}-A_{(i-1)h}\|_\infty\right)^m} \\
& \leq 4^{\rho+\frac{\rho^2}{2}} \mathbb{E}\left[\exp \left\{\left(\frac{2 \nu e^{\beta A_T}}{\nu-1}\right)^2 p \lambda|\xi|\right\}\right]^{\frac{\rho^2}{4}} \mathbb{E}\left[\exp \left\{\left(\frac{2 \nu e^{\beta A_T}}{\nu-1}\right)^2 p \lambda \eta\right\}\right]^{\frac{\rho^2}{4}} \mathbb{E}\left[\exp \left\{\frac{2 \nu p \lambda}{\nu-1} e^{\beta A_T} \eta\right\}\right]^{\frac{\rho}{2}} \\
& \quad \times \mathbb{E}\left[\exp \left\{\frac{2 \nu p \lambda}{\nu-1} e^{\beta A_T}\left|Y_s^{(0)}\right|\right\}\right]^{\frac{\rho}{2} e^{m \beta A_T}\left(\gamma_1+\gamma_2+\beta \|A_{ih}-A_{(i-1)h}\|_\infty\right)^m} \mathbb{E}\left[\exp \left\{p \lambda \sup _{s \in[0, T]}\left|Y_s^{(0)}\right|\right\}\right]^{e^{m \beta A_T}\left(\gamma_1+\gamma_2+\beta \|A_{ih}-A_{(i-1)h}\|_\infty\right)^m},
\end{aligned}
$$
where we used (\ref{eq after iteration}) in the last inequality. Putting the above inequalities together and applying Hölder's inequality again yields that for any $p \geq 2$
\begin{equation}
\label{eq p leq 2}
\begin{aligned}
& \mathbb{E}\left[\exp \left\{p \lambda \sup _{s \in[0, T]}\left|Y_s^{(m)}\right|\right\}\right] \leq \mathbb{E}\left[\exp \left\{2 p \lambda \sup _{s \in[0, T-h]}\left|Y_s^{(m)}\right|\right\}\right]^{\frac{1}{2}} \mathbb{E}\left[\exp \left\{2 p \lambda \sup _{s \in[T-h, T]}\left|Y_s^{(m)}\right|\right\}\right]^{\frac{1}{2}} \\
& \leq 4^{\rho+\frac{\rho^2}{4}} \mathbb{E}\left[\exp \left\{\left(\frac{2 \nu e^{\beta A_T}}{\nu-1}\right)^2 2 p \lambda|\xi|\right\}\right]^{\frac{\rho}{4}+\frac{\rho^2}{8}} \mathbb{E}\left[\exp \left\{\left(\frac{2 \nu e^{\beta A_T}}{\nu-1}\right)^2 2 p \lambda \eta\right\}\right]^{\frac{\rho}{2}+\frac{\rho^2}{8}} \\
& \quad \times \mathbb{E}\left[\exp \left\{\frac{4 \nu}{\nu-1} p \lambda e^{\beta A_T} \sup _{s \in[0, T]}\left|Y_s^{(0)}\right|\right\}\right]^{\left(1+\frac{\rho}{4}\right) e^{m \beta A_T}\left(\gamma_1+\gamma_2+\beta \|A_{ih}-A_{(i-1)h}\|_\infty\right)^m},
\end{aligned}  
\end{equation}
which is uniformly bounded with respect to $m$. Iterating the above procedure $\mu$ times in the general case and recalling [Draft, Prop 4.6], we eventually get
$$
\sup _{m \geq 0} \mathbb{E}\left[\exp \left\{p \lambda \sup _{s \in[0, T]}\left|Y_s^{(m)}\right|\right\}+\left(\int_0^T\int_E\left|U_t(e)^{(m)}\right|^2\phi_t(de)dA_t\right)^{p/2}+\left|K_T^{(m)}\right|^p\right]<\infty, \quad \forall p \geq 1,
$$
which concludes the proof.
\end{proof}

\begin{proof}[Proof of Lemma \ref{estimate 2}]
Without loss of generality, assume $f(t, y, \mu, \cdot)$ is convex. For each fixed $m, q \geq 1$ and $\theta \in(0,1)$, we can define similarly $\delta_\theta \ell^{(m, q)}$ and $\delta_\theta \widetilde{\ell}^{(m, q)}$ for $y^\tau, u^\tau$. Then, the pair of processes $\left(\delta_\theta y^{(m, q), \tau}, \delta_\theta u^{(m, q), \tau}\right)$ satisfies the following BSDE:
\begin{equation}
\label{eq delta}
\delta_\theta y_t^{(m, q), \tau}=\delta_\theta \eta^{(m, q), \tau}+\int_t^\tau\left(\delta_\theta f^{(m, q)}\left(s, \delta_\theta y_s^{(m, q), \tau}, \delta_\theta z_s^{(m, q), \tau}\right)+\delta_\theta f_0^{(m, q)}(s)\right) d A_s-\int_t^\tau \int_E \delta_\theta u_s(e)^{(m, q), \tau} d q(ds, de),
\end{equation}
where the terminal condition and generator are given by
$$
\begin{aligned}
& \delta_\theta \eta^{(m, q)}=\xi \mathbf{1}_{\{\tau=T\}}+\frac{h\left(\tau, Y_\tau^{(m+q-1)},\left(\mathbb{P}_{Y_s^{(m+q-1)}}\right)_{s=\tau}\right)-\theta h\left(\tau, Y_\tau^{(m-1)},\left(\mathbb{P}_{Y_s^{(m-1)}}\right)_{s=\tau}\right)}{1-\theta} \mathbf{1}_{\{\tau<T\}}, \\
& \delta_\theta f_0^{(m, q)}(t)=\frac{1}{1-\theta}\left(f\left(t, y_t^{(m+q), \tau}, \mathbb{P}_{Y_t^{(m+q-1)}}, u_t^{(m+q), \tau}\right)-f\left(t, y_t^{(m+q), \tau}, \mathbb{P}_{Y_t^{(m-1)}}, u_t^{(m+q), \tau}\right)\right), \\
\end{aligned}
$$
$$
\begin{aligned}
\delta_\theta f^{(m, q)}(t, y, u)=\frac{1}{1-\theta}\left( f\left(t, (1-\theta)y+\theta y_t^{(m), \tau}, \mathbb{P}_{Y_t^{(m-1)}}, (1-\theta)u+\theta u_t^{(m), \tau}\right)\right.
\left.-\theta f\left(t, y_t^{(m), \tau}, \mathbb{P}_{Y_t^{(m-1)}},u_t^{(m), \tau}\right)\right). \\
\end{aligned}
$$
Recalling Assumptions (H3) and (H4), we have
$$
\begin{aligned}
& \delta_\theta \eta^{(m, q)} \leq|\xi|+\left|h\left(\tau, 0, \delta_0\right)\right|+\gamma_1(\left.2\left|Y_\tau^{(m-1)}\right|+\left|\delta_\theta Y_\tau^{(m-1, q)}\right|\right) \left.+\gamma_2\left(2 \mathbb{E}\left[\left|Y_s^{(m-1)}\right|\right]_{s=\tau}+\mathbb{E}\left[\left|\delta_\theta Y_s^{(m-1, q)}\right|\right]_{s=\tau}\right)\right), \\
& \delta_\theta f_0^{(m, q)}(t) \leq \beta\left(\mathbb{E}\left[\left|Y_t^{(m+q-1)}\right|\right]+\right.\left.\mathbb{E}\left[\left|\delta_\theta Y_t^{(m-1, q)}\right|\right]\right), \\
& \delta_\theta f^{(m, q)}(t, y, z) \leq \beta|y|+\beta\left|y_t^{(m), \tau}\right|+f\left(t, y_t^{(m), \tau}, \mathbb{P}_{Y_t^{(m-1)}},u\right) \\
& \leq \alpha_t +2 \beta\left|y_t^{(m), \tau}\right|+\beta \mathbb{E}\left[\left|Y_t^{(m-1)}\right|\right]+\beta|y|+\frac{1}{\lambda}j_\lambda(t, u).
\end{aligned}
$$
For any $m, q \geq 1$, set $C_2:=\sup _m \mathbb{E}\left[\sup _{s \in[0, T]}\left|Y_s^{(m)}\right|\right]<\infty$ and
$$
\begin{aligned}
& \zeta^{(m, q)}=e^{\beta A_T}\left(|\xi|+\eta+\gamma_1\left(\sup _{s \in[0, T]}\left|Y_s^{(m-1)}\right|+\sup _{s \in[0, T]}\left|Y_s^{(m+q-1)}\right|\right)+\left(\gamma_2+\beta A_T\right) C_2\right), \\
& \chi^{(m, q)}:=\eta+2 \gamma_1\left(\sup _{s \in[0, T]}\left|Y_s^{(m+q-1)}\right|+\sup _{s \in[0, T]}\left|Y_s^{(m-1)}\right|\right)+2\left(\gamma_2+\beta A_T\right) C_2, \\
& \tilde{\chi}^{(m, q)}:=\chi^{(m, q)}+\sup _{s \in[0, T]}\left|Y_s^{(m+q)}\right|+\sup _{s \in[0, T]}\left|Y_s^{(m)}\right| .
\end{aligned}
$$
Using assertion (ii) of Lemma \ref{priori estimates} to (\ref{eq delta}) and Hölder's inequality, we derive that for any $p \geq 1$
$$
\begin{aligned}
& \exp \left\{p \lambda\left(\delta_\theta y_t^{(m, q), \tau}\right)^{+}\right\} \\
& \leq \mathbb{E}_t\left[\operatorname { e x p } \left\{p \lambda e ^ { \beta A_T} \left(|\xi|+\chi^{(m, q)}+2 \beta A_T \sup _{s \in[t, T]}\left|y_s^{(m), \tau}\right|+\gamma_1 \sup _{s \in[t, T]}\left|\delta_\theta Y_s^{(m-1, q)}\right|\right.\right.\right. \\
&\quad \left.\left.\left.+\left(\gamma_2+\beta(A_T-A_t)\right) \sup _{s \in[t, T]} \mathbb{E}\left[\left|\delta_\theta Y_s^{(m-1, q)}\right|\right]\right)\right\}\right] \\
& \leq \mathbb{E}_t\left[\operatorname { e x p } \left\{2 p \lambda e ^ { \beta A_T} \left(|\xi|+\chi^{(m, q)}+\gamma_1 \sup _{s \in[t, T]}\left|\delta_\theta Y_s^{(m-1, q)}\right|\right.\right.\right. \\
&\quad \left.\left.\left.+\left(\gamma_2+\beta(A_T-A_t)\right) \sup _{s \in[t, T]} \mathbb{E}\left[\left|\delta_\theta Y_s^{(m-1, q)}\right|\right]\right)\right\}\right]^{\frac{1}{2}} \mathbb{E}_t\left[\exp \left\{4 p \lambda e^{2 \beta A_T} \sup _{s \in[t, T]}\left|y_s^{(m), \tau}\right|\right\}\right]^{\frac{1}{2}} .
\end{aligned}
$$
Recalling (\ref{eq estimate y}) and using Doob's maximal inequality, we conclude that for each $p \geq 2$ and $t \in[0, T]$
$$
\mathbb{E}_t\left[\exp \left\{p \lambda \sup _{s \in[t, T]}\left|y_s^{(m), \tau}\right|\right\}\right] \vee \mathbb{E}_t\left[\exp \left\{p \lambda \sup _{s \in[t, T]}\left|y_s^{(m+q), \tau}\right|\right\}\right] \leq 4 \mathbb{E}_t\left[\exp \left\{p \lambda \zeta^{(m, q)}\right\}\right], \quad \forall m, q \geq 1 .
$$
It follows from (\ref{eq representation}) that
$$
\begin{aligned}
& \exp \left\{p \lambda\left(\delta_\theta Y_t^{(m, q)}\right)^{+}\right\} \leq \underset{\tau \in \mathcal{T}_t}{\operatorname{ess} \sup } \exp \left\{p \lambda\left(\delta_\theta y_t^{(m, q), \tau}\right)^{+}\right\} \\
& \leq 4 \mathbb{E}_t\left[\exp \left\{2 p \lambda e^{\beta A_T}\left(|\xi|+\chi^{(m, q)}+\gamma_1 \sup _{s \in[t, T]}\left|\delta_\theta Y_s^{(m-1, q)}\right|+\left(\gamma_2+\beta(A_T-A_t)\right) \sup _{s \in[t, T]} \mathbb{E}\left[\left|\delta_\theta Y_s^{(m-1, q)}\right|\right]\right)\right\}\right]^{\frac{1}{2}} \\
& \quad \times \mathbb{E}_t\left[\exp \left\{4 p \lambda e^{2 \beta A_T} \zeta^{(m, q)}\right\}\right]^{\frac{1}{2}} .
\end{aligned}
$$
Using a similar method, we derive that
$$
\begin{aligned}
& \exp \left\{p \lambda\left(\delta_\theta \widetilde{Y}_t^{(m, q)}\right)^{+}\right\} \\
& \leq 4 \mathbb{E}_t\left[\exp \left\{2 p \lambda e^{\beta A_T}\left(|\xi|+\chi^{(m, q)}+\gamma_1 \sup _{s \in[t, T]}\left|\delta_\theta \widetilde{Y}_s^{(m-1, q)}\right|+\left(\gamma_2+\beta(A_T-A_t)\right) \sup _{s \in[t, T]} \mathbb{E}\left[\left|\delta_\theta \widetilde{Y}_s^{(m-1, q)}\right|\right]\right)\right\}\right]^{\frac{1}{2}} \\
& \quad \times \mathbb{E}_t\left[\exp \left\{4 p \lambda e^{2 \beta A_T} \zeta^{(m, q)}\right\}\right]^{\frac{1}{2}} .
\end{aligned}
$$
According to the fact that
$$
\left(\delta_\theta Y^{(m, q)}\right)^{-} \leq\left(\delta_\theta \widetilde{Y}^{(m, q)}\right)^{+}+2\left|Y^{(m)}\right| \text { and }\left(\delta_\theta \widetilde{Y}^{(m, q)}\right)^{-} \leq\left(\delta_\theta Y^{(m, q)}\right)^{+}+2\left|Y^{(m+q)}\right|,
$$
we deduce that
$$
\begin{aligned}
& \exp \left\{p \lambda\left|\delta_\theta Y_t^{(m, q)}\right|\right\} \vee \exp \left\{p \lambda\left|\delta_\theta \widetilde{Y}_t^{(m, q)}\right|\right\} \\
& \leq \exp \left\{p \lambda\left(\left(\delta_\theta Y_t^{(m, q)}\right)^{+}+\left(\delta_\theta \widetilde{Y}_t^{(m, q)}\right)^{+}+2\left|Y_t^{(m)}\right|+2\left|Y_t^{(m+q)}\right|\right)\right\} \\
& \leq 4^2 \mathbb{E}_t\left[\exp \left\{2 p \lambda e^{\beta A_T}\left(|\xi|+\tilde{\chi}^{(m, q)}+\gamma_1 \sup _{s \in[t, T]} \delta_\theta \bar{Y}_s^{(m-1, q)}+\left(\gamma_2+\beta(A_T-A_t)\right) \sup _{s \in[t, T]} \mathbb{E}\left[\delta_\theta \bar{Y}_s^{(m-1, q)}\right]\right)\right\}\right] \\
& \quad \times \mathbb{E}_t\left[\exp \left\{4 p \lambda e^{2 \beta A_T} \zeta^{(m, q)}\right\}\right] .
\end{aligned}
$$
Applying Doob's maximal inequality and Hölder's inequality, we get that for each $p>1$ and $t \in[0, T]$
$$
\begin{aligned}
& \mathbb{E}\left[\exp \left\{p \lambda \sup _{s \in[t, T]} \delta_\theta \bar{Y}_s^{(m, q)}\right\}\right] \\
& \leq 4^5 \mathbb{E}\left[\exp \left\{4 p \lambda e^{\beta A_T} \widetilde{\nu}\left(|\xi|+\widetilde{\chi}^{(m, q)}+\gamma_1 \sup _{s \in[t, T]} \delta_\theta \bar{Y}_s^{(m-1, q)}+\left(\gamma_2+\beta(A_T-A_t)\right) \sup _{s \in[t, T]} \mathbb{E}\left[\delta_\theta \bar{Y}_s^{(m-1, q)}\right]\right)\right\}\right]^{\frac{1}{\nu}} \\
& \quad \times \mathbb{E}\left[\exp \left\{\frac{4 \widetilde{\nu}}{\widetilde{\nu}-1} p \lambda e^{2 \beta A_T} \zeta^{(m, q)}\right\}\right]^{\frac{\tilde{\tilde{\nu}}-1}{\nu}} .
\end{aligned}
$$
Recalling the definitions of $h, \nu$ and $\widetilde{\nu}$, we have
$$
\begin{aligned}
& \mathbb{E}\left[\exp \left\{p \lambda \sup _{s \in[T-h, T]} \delta_\theta \bar{Y}_s^{(m, q)}\right\}\right] \\
& \leq 4^5 \mathbb{E}\left[\exp \left\{\frac{8 \nu \widetilde{\nu} p \lambda}{\nu-1} e^{\beta A_T}|\xi|\right\}\right]^{\frac{\nu-1}{2 \nu \nu}} \mathbb{E}\left[\exp \left\{\frac{8 \nu \widetilde{\nu} p \lambda}{\nu-1} e^{\beta A_T} \widetilde{\chi}^{(m, q)}\right\}\right]^{\frac{\nu-1}{2 \nu \nu}} \mathbb{E}\left[\exp \left\{\frac{4 \widetilde{\nu}}{\widetilde{\nu}-1} p \lambda e^{2 \beta A_T} \zeta^{(m, q)}\right\}\right]^{\frac{\tilde{\nu}-1}{\nu}} \\
& \quad \times \mathbb{E}\left[\exp \left\{p \lambda \sup _{s \in[T-h, T]} \delta_\theta \bar{Y}_s^{(m-1, q)}\right\}\right]^{4 e^{\beta A_T}\left(\gamma_1+\gamma_2+\beta \|A_{ih}-A_{(i-1)h}\|_\infty\right)} . \\
&
\end{aligned}
$$
Set $\widetilde{\rho}=\frac{1}{1-4 e^{\beta A_T}\left(\gamma_1+\gamma_2+\beta \|A_{ih}-A_{(i-1)h}\|_\infty\right)}$. If $\mu=1$, for each $p \geq 1$ and $m, q \geq 1$
$$
\begin{aligned}
& \mathbb{E}\left[\exp \left\{p \lambda \sup _{s \in[0, T]} \delta_\theta \bar{Y}_s^{(m, q)}\right\}\right] \\
& \leq 4^{5 \tilde{\rho}} \mathbb{E}\left[\exp \left\{\frac{8 \nu \widetilde{\nu} p \lambda}{\nu-1} e^{\beta A_T}|\xi|\right\}\right]^{\frac{\tilde{\rho}}{2}} \sup _{m, q \geq 1} \mathbb{E}\left[\exp \left\{\frac{8 \nu \widetilde{\nu} p \lambda}{\nu-1} e^{\beta A_T} \widetilde{\chi}^{(m, q)}\right\}\right]^{\frac{\tilde{\tilde{p}}}{2}} \sup _{m, q \geq 1} \mathbb{E}\left[\exp \left\{\frac{4 \widetilde{\nu}}{\widetilde{\nu}-1} p \lambda e^{2 \beta A_T} \zeta^{(m, q)}\right\}\right]^{\tilde{\rho}} \\
& \quad \times \mathbb{E}\left[\exp \left\{p \lambda \sup _{s \in[0, T]} \delta_\theta \bar{Y}_s^{(1, q)}\right\}\right]^{e^{(m-1) \beta A_T}\left(4 \gamma_1+4 \gamma_2+4 \beta \|A_{ih}-A_{(i-1)h}\|_\infty\right)^{m-1}} .
\end{aligned}
$$
Applying Lemma \ref{estimate 1}, we have for any $\theta \in(0,1)$
$$
\lim _{m \rightarrow \infty} \sup _{q \geq 1} \mathbb{E}\left[\exp \left\{p \lambda \sup _{s \in[0, T]} \delta_\theta \bar{Y}_s^{(1, q)}\right\}\right]^{e^{(m-1) \beta A_T}\left(4 \gamma_1+4 \gamma_2+4 \beta \|A_{ih}-A_{(i-1)h}\|_\infty\right)^{m-1}}=1.
$$
It follows that
$$
\begin{aligned}
& \sup _{\theta \in(0,1)} \lim _{m \rightarrow \infty} \sup _{q \geq 1} \mathbb{E}\left[\exp \left\{p \lambda \sup _{s \in[0, T]} \delta_\theta \bar{Y}_s^{(m, q)}\right\}\right] \\
& \leq 4^{5 \widetilde{\rho}} \mathbb{E}\left[\exp \left\{\frac{8 \nu \widetilde{\nu} p \lambda}{\nu-1} e^{\beta A_T}|\xi|\right\}\right]^{\frac{\tilde{\rho}}{2}} \sup _{m, q \geq 1} \mathbb{E}\left[\exp \left\{\frac{8 \nu \widetilde{\nu} p \lambda}{\nu-1} e^{\beta A_T} \widetilde{\chi}^{(m, q)}\right\}\right]^{\frac{\tilde{\rho}}{2}} \\
& \quad \times \sup _{m, q \geq 1} \mathbb{E}\left[\exp \left\{\frac{4 \widetilde{\nu}}{\widetilde{\nu}-1} p \lambda e^{2 \beta A_T} \zeta^{(m, q)}\right\}\right]^{\widetilde{\rho}}<\infty . \\
&
\end{aligned}
$$
If $\mu=2$, proceeding identically as to derive (\ref{eq p leq 2}), we have for any $p \geq 1$
$$
\begin{aligned}
& \mathbb{E}\left[\exp \left\{p \lambda \sup _{s \in[0, T]} \delta_\theta \bar{Y}_s^{(m, q)}\right\}\right] \\
& \leq 4^{5 \widetilde{\rho}+\frac{5 \tilde{\rho}^2}{4}} \mathbb{E}\left[\exp \left\{\left(\frac{8 \nu \widetilde{\nu} e^{\beta A_T}}{\nu-1}\right)^2 2 p \lambda|\xi|\right\}\right]^{\frac{\tilde{\rho}}{4}+\frac{\tilde{\rho}^2}{8}} \sup _{m, q \geq 1} \mathbb{E}\left[\exp \left\{\left(\frac{8 \nu \widetilde{\nu} e^{\beta A_T}}{\nu-1}\right)^2 2 p \lambda \widetilde{\chi}^{(m, q)}\right\}\right]^{\frac{\tilde{\rho}}{2}+\frac{\tilde{\rho}^2}{8}} \\
& \quad \times \sup _{m, q \geq 1} \mathbb{E}\left[\exp \left\{\frac{32 \nu \widetilde{\nu}^2 e^{\beta A_T}}{(\widetilde{\nu}-1)(\nu-1)} p \lambda e^{2 \beta A_T} \zeta^{(m, q)}\right\}\right]^{\widetilde{\rho}+\frac{\tilde{\rho}^2}{4}} \\
& \quad \times \mathbb{E}\left[\exp \left\{\frac{8 \nu \widetilde{\nu} e^{\beta A_T}}{\nu-1} p \lambda \sup _{s \in[0, T]}\left|\delta_\theta \bar{Y}_s^{(1, q)}\right|\right\}\right]^{\left(1+\frac{\tilde{\rho}}{4}\right) e^{(m-1) \beta A_T}\left(4 \gamma_1+4 \gamma_2+4 \beta \|A_{ih}-A_{(i-1)h}\|_\infty\right)^{m-1}},
\end{aligned}
$$
which also implies the desired assertion in this case. Iterating the above procedure $\mu$ times in the general case, we get the desired result.
\end{proof}

\end{appendices}

\bibliographystyle{plain}
\bibliography{RBSDEJ}

\end{document}